\documentclass[12pt]{amsart}
\usepackage{amssymb}
\newcommand{\gl}{\mathfrak{gl}}
\newcommand{\g}{\mathfrak{g}}
\newcommand{\h}{\mathfrak{h}}
\newcommand{\C}{\mathbb{C}}

\newcommand{\OCat}{\mathcal{O}}
\newcommand{\p}{\mathfrak{p}}
\newcommand{\Hecke}{\mathcal{H}}
\newcommand{\Cat}{\mathcal{C}}
\newcommand{\bs}{\underline{s}}
\newcommand{\param}{\mathfrak{P}}
\renewcommand{\sl}{\mathfrak{sl}}
\newcommand{\Ext}{\operatorname{Ext}}

\newcommand{\Z}{\mathbb{Z}}
\newcommand{\cont}{\operatorname{cont}}
\newcommand{\dotimes}{\dot{\otimes}}
\newcommand{\End}{\operatorname{End}}
\newcommand{\Hom}{\operatorname{Hom}}
\newcommand{\wt}{\operatorname{wt}}
\newcommand{\Par}{\mathcal{P}}
\newcommand{\Ind}{\operatorname{Ind}}
\newcommand{\Res}{\operatorname{Res}}
\unitlength=1mm   \numberwithin{equation}{section}
\newtheorem{Thm}{Theorem}[section]
\newtheorem{Prop}[Thm]{Proposition}
\newtheorem{Cor}[Thm]{Corollary}
\newtheorem{Lem}[Thm]{Lemma}
\theoremstyle{definition}

\newtheorem{defi}[Thm]{Definition}
\newtheorem{Rem}[Thm]{Remark}
\newtheorem{Conj}[Thm]{Conjecture}

\title{Proof of Varagnolo-Vasserot conjecture on cyclotomic categories $\mathcal{O}$}
\author{Ivan Losev}
\address{Department
of Mathematics, Northeastern University, Boston MA 02115 USA}
\email{i.loseu@neu.edu}
\begin{document}
\begin{abstract}
We prove an asymptotic version of a conjecture by Varagnolo and Vasserot on an equivalence
between the category $\mathcal{O}$ for a cyclotomic Rational Cherednik algebra and a suitable
truncation of an affine parabolic category $\mathcal{O}$ that, in particular, implies  Rouquier's
conjecture on the decomposition numbers in the former. Our proof uses two ingredients:
an extension of Rouquier's deformation approach as well as categorical actions on highest weight
categories and related combinatorics.
\end{abstract}
\maketitle
\tableofcontents
\section{Introduction}
Rational Cherednik algebras were introduced by Etingof and Ginzburg, \cite{EG}. These are associative algebras over $\C$
constructed from a complex reflection group, say $W$, and depending on a parameter, say $p$, that is a collection
of complex numbers. They have many things in common
with the universal enveloping algebras of semisimple Lie algebras, in particular, they have a triangular decomposition.
This allows one to define the categories $\mathcal{O}$ for such algebras, this was done in \cite{GGOR}. There are analogs
of Verma modules, parameterized by irreducible $W$-modules,
and an ordering on the set of simples in the GGOR category $\mathcal{O}$ making it into
a highest weight category. So there is a basic question one can ask: compute the multiplicity of
a given simple module in a given standard (=Verma) module.

The nicest and, perhaps, most important family of complex reflection groups is $W=G(\ell,1,n):=S_n\ltimes (\Z/\ell\Z)^n$,
where $n,\ell$ are positive integers. This group acts on $\C^n$ by permutations of coordinates followed by multiplications by roots of $1$ of order $\ell$. There are more general infinite families, the groups $G(\ell,r,n)$, where $r$ is a divisor
of $\ell$, but the study of the corresponding categories $\mathcal{O}$ can be, to some extent, reduced
to the case of $G(\ell,1,n)$ and this is one of the reasons why our case is important. Another reason is that
the corresponding category has an additional interesting structure that is not present in the other cases,
a categorical Kac-Moody action to be recalled below. Yet another reason is a connection to the geometry of
symplectic resolutions of quotient singularities.

A significant progress in determining the multiplicities was made by Rouquier in \cite{rouqqsch}, where he computed
the multiplicities in the case $\ell=1$ and made a conjecture for  all $\ell$ (the conjecture
was made for some special, but, in a sense, the most interesting and ``non-degenerate'' values of $p$).
The conjecture says that the multiplicities are given by certain parabolic Kazhdan-Lusztig polynomials.
The techniques used in the proof for $\ell=1$ were roughly as
follows. In \cite{GGOR} the authors introduced a so called KZ functor  from the Cherednik
category $\mathcal{O}$ to the category of modules over the Hecke algebra  $\mathcal{H}$ of $W$ with parameters
recovered from $p$. This is a quotient functor. Rouquier developed techniques that allow
to check when two highest weight categories  admitting quotient functors to $\mathcal{H}$-$\operatorname{mod}$
are equivalent. For $\ell=1$ there is another category with a nice quotient functor, the category
of modules over an appropriate $q$-Schur algebra that was shown to be equivalent to the Cherednik category
$\mathcal{O}$ (under a certain ``faithfulness'' condition on the parameters).

For $\ell>1$, the situation is more complicated. For certain, so to say, ``dominant'' and ``faithful'', values of $p$
Rouquier proved in \cite{rouqqsch} that the Cherednik category $\mathcal{O}$ is equivalent to the category
of modules over a suitable cyclotomic $q$-Schur algebra of Dipper, James and Mathas. The multiplicities
for the latter categories were recently computed by Stroppel and Webster, \cite{SW}.

On the other hand, Varagnolo and Vasserot in \cite{VV} produced another category, where the multiplicities
were shown to be as required by the Rouquier conjecture.
Their category is a certain truncation of an affine parabolic
category $\mathcal{O}$. They conjectured an equivalence of that category with the
Cherednik category $\mathcal{O}$.

The goal of this paper is to prove that conjecture (in a somewhat weaker form that is still
sufficient for checking the Rouquier conjecture). Together with earlier results of Shan, Varagnolo
and Vasserot, our result also implies a conjecture of Chuang
and Miyashi, \cite{CM}, claiming that the Cherednik category $\mathcal{O}$ is Koszul
and describing the Koszul dual.

\subsection{Ideas of proof}
Our proof of the Varagnolo-Vasserot conjecture  uses two groups of ideas. First, we use deformation ideas
initially due to Rouquier, \cite{rouqqsch}, with  further extensions. Some of them are due to Rouquier, Shan, Varagnolo,
Vasserot and some are to be developed in the present paper. Second, to properly implement these ideas we need
categorical actions on highest weight categories, a topic initiated by the author in \cite{cryst},\cite{str}
and further developed by the author and Webster in \cite{LW} and in the present paper.

Let us describe the deformation ideas. The GGOR category $\mathcal{O}$ admits a quotient functor
(the KZ functor of \cite{GGOR} to be reviewed in Section \ref{SS_KZ}) to the category of modules
over a cyclotomic quotient of the affine Hecke
algebra. This functor is fully faithful on certain subcategories:
for example, on the categories of tilting
and of projective objects, \cite[Theorem 5.3]{GGOR}.
Also it is fully faithful on the whole category of standardly filtered objects
({\it $0$-faithful} in Rouquier's terminology) under some restrictions on the parameters for the Cherednik
algebra, see \cite[Proposition 5.9]{GGOR}. As Rouquier checked in \cite[Proposition 4.42]{rouqqsch}
this implies that, after a generic
one-parameter deformation of the categories
of interest, the KZ functor becomes {\it 1-faithful} (i.e., an isomorphism on $\Hom$ and $\Ext^1$
between standardly filtered objects). Two highest weight
categories over $\C[[\hbar]]$  with quotient functors to the same category are equivalent provided their orders
are the same, the quotient functors are 1-faithful and are equivalences over $\C((\hbar))$,
\cite[Theorem 4.49]{rouqqsch}. So the problem is to establish an analog of the KZ functor for a truncated affine parabolic category $\mathcal{O}$. To produce a functor is not
difficult, this is done using categorical Kac-Moody actions, we define a projective
object representing the functor in Proposition \ref{Prop:FnDelta} below.
What is much harder is to prove faithfulness properties.
Recently, Rouquier, Shan, Varagnolo and Vasserot proposed to consider 2-parametric deformations and announced
that  0-faithfulness in points of codimension 1 yields 1-faithfulness for the deformed categories.
Considering 2-parametric deformations is one extension of the original technique of Rouquier that we will use.
We will see, Theorem \ref{Thm:Thm_equi} and Proposition \ref{Prop:Cher_proj_emb},
that it is enough to show that the quotient functor from the deformed affine category $\mathcal{O}$
is only 0-faithful. The 0-faithfulness condition follows from checking (-1)-faithfulness in
codimension $1$, Proposition \ref{Prop:faith}.

There is one  more significant extension of Rouquier's technique that we use.
We bypass the problem that sometimes the quotient functors to the cyclotomic
Hecke categories are  not 0-faithful by considering larger quotients described in
Section \ref{SS_ext_quot_setting}. The main result of that section is that
the larger quotients of the truncated affine category and of the Cherednik category
are equivalent. Modulo checking the faithfulness properties of the quotient
functor from the affine category $\mathcal{O}$, this yields a proof of the
Varagnolo-Vasserot conjecture.

Let us explain how the theory of categorical actions on highest weight categories comes into play.
Results of Rouquier, \cite{Rouquier_2Kac}, see, in particular, Corollary 5.7 there,
suggest a way to produce a quotient functor to
a cyclotomic Hecke category (i.e., the direct sum over all $n$ of the module categories over cyclotomic
Hecke algebra with fixed parameters and $n$ variables)
from some category $\mathcal{C}$. Namely, one gets such a functor if $\mathcal{C}$
is equipped with a categorical action
of $\hat{\sl}_e$ that categorifies an integrable   $\hat{\sl}_e$-module with weights
bounded from above. There is a categorical
action on the affine parabolic category $\mathcal{O}$ before the truncation: this is provided
by the Kazhdan-Lusztig tensor products. This action does not restrict to the truncated category
in a straightforward way (as the truncated category is not stable under the categorification functors).
However one can still define a (``restricted'', but this is not of importance) categorical action
on the truncated category using the categorical splitting techniques from \cite{str}, this will be
done in Section \ref{SS_trunc_cat}. This produces a required quotient functor,
Section \ref{SS_proj}. Further,  using structural results
obtained in \cite{str}, one can reduce the study of the faithfulness properties for this functor to some purely combinatorial questions concerning  crystal structures on the multipartitions. More precisely, there is a combinatorial condition that guarantees $(-1)$-faithfulness  of the quotient functor, see Section \ref{SS_comb_to_-1faith}.
The combinatorial condition is already sufficient to checking the (-1)-faithfulness in codimension $1$
for the affine parabolic category. This completes the proof of the Varagnolo-Vasserot conjecture.


\begin{Rem}
We want to indicate the dependence of the present paper on a related work. We  use an idea  due to
Rouquier, Shan, Varagnolo and Vasserot explained before.  This idea was mentioned in Shan's talk
in Luminy in July 2012 without explanations on how to make it to work, and the paper, \cite{RSVV},
appeared when our paper was ready.
There is also a related work of Webster, \cite{Webster_new}, where he
proves an equivalence between the GGOR category and a certain diagrammatic category.
\end{Rem}

\begin{Rem}
The version of this paper that appeared in 2013 had a serious gap. Presumably, the gap can be fixed
using Zuckerman functors for affine parabolic categories $\mathcal{O}$, however the fix is by no means
easy. In July 2015, we have discovered Theorem \ref{Thm:Thm_equi} that allows to significantly simplify
the original proof.
\end{Rem}

\subsection{Structure of the paper}
In  Section \ref{S_cat} we describe the highest weight categories we consider: the categories $\mathcal{O}$
for cyclotomic Rational Cherednik algebra and affine parabolic categories $\mathcal{O}$ both in the undeformed
and deformed settings. We also recall basic combinatorics of these categories. This section contains no new results.

In Section \ref{S_faith} we provide general results on faithfulness properties of quotient functors from highest
weight categories. The main results of this section are Proposition \ref{Prop:faith}
(that is a version of \cite[Proposition 4.42]{rouqqsch}), Theorem \ref{Thm:Thm_equi} and Proposition \ref{Prop:Cher_proj_emb}.

Section \ref{S_cat_KM} deals with categorical Kac-Moody actions on highest weight categories. It defines categorical
type A Kac-Moody actions and recalls results from \cite{cryst},\cite{str}. There are no new results
there.

Section \ref{S_trunc_cat} is new. There we equip the truncated affine parabolic category $\mathcal{O}$
with a restricted type A categorical Kac-Moody action and so produce a functor to the cyclotomic Hecke
category.

In Section \ref{S_faith_comb} we study an interplay between the faithfulness properties of quotient morphisms
and combinatorial properties of crystals. Namely, we state a combinatorial condition that guarantees  vanishing
of $\operatorname{Hom}$  from a suitable simple to a suitable tilting.  Finally, we check that our combinatorial
condition holds in a certain special case.


Section \ref{S_C_cat_equi} we define new quotient functors that are ``larger'' than the functors considered before
(our old functors factor through new ones). Then we show that the target categories for our new functors
in the GGOR and in the parabolic setting are equivalent.


Finally, in the last section of this paper, we complete the proof of the main equivalence theorem
that yields an asymptotic version of the Varagnolo-Vasserot conjecture.

The paper contains an appendix that provides an independent proof for $\ell=1$.

{\bf Acknowledgements}. My research was supported by the NSF grants DMS-0900907, DMS-1161584.
This paper would not have appeared without numerous conversations with
R. Bezrukavnikov, I. Gordon, B. Webster. I am very grateful to them. I also would like to thank
J. Brundan, D. Gaitsgory,  P. Etingof, P. Shan for useful discussions. Special thanks
are to E. Vasserot for a stimulating e-mail correspondence. Finally, I would like to thank
the referee for the many comments that helped me to improve the exposition.

\section{Categories of interest}\label{S_cat}
\subsection{Poset of multipartitions}\label{SS_part_poset}
Let $\ell$ be a positive integer.  We consider the set $\mathcal{P}_\ell$ of $\ell$-multipartitions, i.e., $\ell$-tuples $(\lambda^{(1)},\ldots,\lambda^{(\ell)})$, where $\lambda^{(i)}$ is a partition. We write $|\lambda|$ for the number partitioned by $\lambda$.

A partition can be thought as a Young diagram -- a shape on the coordinate  plane consisting of unit square  boxes.
The diagram corresponding to a partition $\mu$, by definition, consists of squares whose top right corner
has coordinates $(y,x)$ with $0\leqslant y \leqslant \mu_x$. So a box in a multipartition $\lambda$
is given by a triple $(x,y,i)$, where $i=1,\ldots,\ell$ is the number of a multipartition, where the box occurs, and $(x,y)$ are its coordinates: $x$ is the row number, and $y$ is the column number.

We are going to equip $\mathcal{P}_\ell$ with a partial order. This partial order will depend on an integer
$e>1$ and an $\ell$-tuple of integers (a multi-charge) $(s_1,\ldots,s_{\ell})$.
To a box $b=(x,y,i)$ we assign its {\it shifted content} $\cont(b)=y-x+s_i$.

We say that boxes $b,b'$ are equivalent and write $b\sim b'$
if $\cont(b)-\cont(b')$ is divisible by $e$. Also to a box $b=(x,y,i)$ we assign the number
$d(b)= -\frac{\ell}{e} \cont(b)-i$. We write $b\preceq b'$ if $b\sim b'$ and $d(b)- d(b')$
is a non-negative integer. Equivalently, $b\preceq b'$ if $\cont(b')-\cont(b)\in e \Z_{>0}$ or
$\cont(b)=\cont(b')$ and $i<i'$. For two $\lambda,\mu\in \mathcal{P}_{\ell}$ we write $\lambda\preceq \mu$ if $|\lambda|=|\mu|$ and we can number boxes $b_1,\ldots,b_n$ of $\lambda$ and $b_1',\ldots,b_n'$ of $\mu$ in such a way that $b_i\preceq b_i'$ for all $i$. It is not difficult to see that $\lambda\preceq \mu, \mu\preceq \lambda$ actually implies that $\lambda=\mu$.

\subsection{GGOR category $\mathcal{O}$}\label{SS_Cher_O}
Let $\ell,n$ be  positive integers. Consider the finite group $G_n:=\mathfrak{S}_n\ltimes (\Z/\ell\Z)^n$. Let $V$
be its reflection representation (of dimension $n$ for $\ell>1$ and of dimension $n-1$ for $\ell=1$).
Let $\kappa$ be a complex number and $\underline{s}=(s_1,\ldots,s_{\ell})$ be a collection of complex numbers
defined up to a common summand. The rational Cherednik algebra $H_{\kappa,\bs}(n)$ is the quotient
of $T(V\oplus V^*)\rtimes G_n$ by the relations of the form $[x,x']=0, [y,y']=0, [y,x]=w_{x,y}$ for $x,x'\in V^*,
y,y'\in V^*$, where $w_{x,y}$ is an element of $\C G_n$ depending linearly on $x,y$ and
$\kappa, \kappa s_i-\frac{i}{\ell}$ for $ i=1,\ldots,\ell$. The reader is referred to, say \cite[Section 1.2]{VV} for the particular form of the relations. What is important for us is that there is a triangular decomposition $H_{\kappa,\bs}(n)=S(V)\otimes \C G_n\otimes S(V^*)$.

We will consider the category $\mathcal{O}_{\kappa,\bs}(n)$ of $H_{\kappa,\bs}(n)$-modules introduced in
\cite[Section 3.2]{GGOR}. By definition, it consists of all $H_{\kappa,\bs}(n)$-modules that are finitely generated
over $S(V^*)$ and where the action of $V\subset H_{\kappa,\bs}(n)$ is locally nilpotent. This category
has analogs of Verma modules: $\Delta(E)=H_{\kappa,\bs}(n)\otimes_{S(V)\rtimes G_n}E$, where $E$
is an irreducible $G_n$-module. There is a natural identification of the set of irreducible $G_n$-modules
with the set of $\ell$-multipartitions of $n$. Our convention here is almost like in \cite[Section 3.5]{cryst} (with the index $0$
replaced by $\ell$). Consider the direct sum $\mathcal{O}_{\kappa,\bs}:=\bigoplus_{n=0}^{+\infty} \mathcal{O}_{\kappa,\bs}(n)$. This is a highest weight category with poset $\mathcal{P}_\ell$ introduced in Section \ref{SS_part_poset}, the standard objects are $\Delta(\lambda)$. The claim that $\mathcal{O}_{\kappa,\bs}$ is highest weight with respect to a finer (c-function) ordering was already in \cite[Section 3.1]{GGOR}. The claim that the coarser ordering also works follows from  \cite{Griffeth}, see also the proof of \cite[Theorem 1.2]{DG}.
For reader's convenience let us recall the definition of a highest weight category.

An artinian abelian category $\Cat$ equipped with a collection of objects $\Delta(\lambda)$ indexed with  elements
of a poset $\Lambda$ is said to be {\it highest weight} if the following axioms hold.
\begin{itemize}
\item[(HW1)] If $\Hom_{\Cat}(\Delta(\lambda),\Delta(\mu))=0$, then $\lambda\leqslant\mu$, and $\End_{\Cat}(\Delta(\lambda))=\C$.
Moreover, the heads $L(\lambda)$ of $\Delta(\lambda)$ are simple and form a complete list of simple objects in
$\Cat$.
\item[(HW2)] For each $\lambda\in \Lambda$  there is an indecomposable projective object $P(\lambda)$
equipped with a filtration $P(\lambda)=F_0\supset F_1\supset F_2\ldots$
such that $F_0/F_1=\Delta(\lambda)$ and $F_i/F_{i+1}=\Delta(\lambda_i)$ with $\lambda_i>\lambda$ for all $i>0$.
\end{itemize}

We will need a deformation of $\OCat_{\kappa,{\bs}}$. Let $\tilde{\param}$ be the  space $\{(x_0,x_1,\ldots,x_\ell)\}/\{(0,t,t\ldots,t)\}$,  the space of parameters for the Cherednik
algebra. Let $p\in \tilde{\param}$ be the point with coordinates $(\kappa,s_1,\ldots,s_\ell)$. For $\ell>1$, we pick a general 2-dimensional affine subspace $\param$ through $p$ and consider the completion $R:=\C[\param]^{\wedge_p}$
of $\C[\param]$ at $p$.
Then we can form the algebra $H_{\kappa,\bs,R}(n)$ that is the quotient of $T(V\oplus V^*)\rtimes G_n\otimes R$ by the relations corresponding to $(x_0,\ldots,x_\ell)$. We can still define the category $\mathcal{O}_{\kappa, \bs, R}$ in the same way as above. This is an integral (over $R$) highest weight category in the sense of Rouquier, \cite[Section 4.1]{rouqqsch}. For $\ell=1$, we take $\param:=\tilde{\param}$ and define $R, \OCat_{\kappa,R}$, etc., in an analogous
way.

\subsection{Poset of parabolic highest weights}\label{SS_parab_poset}
We fix integers $e>1$ and $\underline{s}:=(s_1,\ldots,s_{\ell})$ with $s_i\geqslant 0$. Set $m:=s_1+\ldots+s_{\ell}$.

Let $\Z^{\underline{s}}$ stand for the set of all $m$-tuples $(a_1,\ldots,a_m)$ of integers such that
$a_1>a_2>\ldots>a_{s_1}, a_{s_1+1}>\ldots>a_{s_1+s_2},\ldots, a_{s_1+\ldots+s_{\ell-1}+1}>\ldots>a_m$.
We are going to equip $\Z^{\underline{s}}$ with two poset structures, one refining the other.

Our coarser poset structure comes from the linkage ordering on a parabolic affine category $\mathcal{O}$
to be considered later. Set $\g:=\gl_m$. Then we can form the affine algebra $\hat{\g}=\g[t^{\pm 1}]\oplus \C c$
and the extended affine Lie algebra $\tilde{\g}=\hat{\g}\oplus \C d$. Let $\h$ denote the Cartan subalgebra of $\g$ consisting of the diagonal matrices and $\hat{\h}:=\h\oplus C c, \tilde{\h}:=\hat{\h}\oplus \C d$. Let $\epsilon_1,\ldots,\epsilon_m$ be a natural basis of $\h^*$ corresponding to the matrix units.
Further, let $\delta\in \tilde{\h}^*$ denote the indecomposable positive imaginary root.

The Weyl group of $\tilde{\g}$, that is, the affine
symmetric group $\hat{\mathfrak{S}}_m=\mathfrak{S}_m\ltimes Q$, where $Q$ is the root lattice of $\g$,
acts naturally on $\Z^{m}$. In particularly, for a real root $\beta=\epsilon_i-\epsilon_j+ n\delta, i\neq j,$ we have $\sigma_\beta (a_1,\ldots,a_m)= (a_1',\ldots,a_m')$, where $a_k'=a_k$ if $k\neq i,j$, $a_i'=a_j+en, a_j'=a_i-en$.
Note that we can embed $\Z^m$ into the weight lattice for $\tilde{\g}$ by
$$A\mapsto \alpha_A:=\sum_{i=1}^m a_i \epsilon_i - e\omega_0+\frac{a_1^2+a_2^2+\ldots+a_m^2}{2e}\delta,$$
where $\omega_0$ is the fundamental weight corresponding to the simple root $\alpha_0=\epsilon_m-\epsilon_0+\delta$.
The map $A\mapsto \alpha_A$ is $\hat{\mathfrak{S}}_m$-equivariant.

We say that an element $A=(a_1,\ldots,a_m)$ is $\bs$-regular if the numbers $a_1,\ldots,a_{s_1}$
are pairwise different, the numbers $a_{s_1+1},\ldots,a_{s_1+s_2}$ are pairwise different, etc.
For an $\bs$-regular element $A$ let $A_+$ denote a unique element  of $\Z^{\bs}$ that is obtained from $A$ by applying a permutation from $\mathfrak{S}_{s_1}\times \mathfrak{S}_{s_2}\times\ldots\times \mathfrak{S}_{s_\ell}\subset \mathfrak{S}_m$. We say that $A>A'$ if there are elements $A_0=A, A_1,\ldots,A_{k-1},A_k=A'$ such that $A_{i}=(\sigma_{\beta_i}A_{i-1})_+$ for some real root $\beta_i$ and $\alpha_{A_{i-1}}-\alpha_{A_i}$ is a nonzero linear combination of positive affine roots with nonzero coefficients. Below we will need an easy lemma describing some properties of  this  ordering.

\begin{Lem}\label{Lem:ordering}
Let $A\in \Z^{\bs}$ and $\beta$ be a positive real root, $\beta=\epsilon_i-\epsilon_j+ n\delta$, where $n\geqslant 0$
if $i<j$ and $n>0$ if $i>j$. Suppose that $(\sigma_{\beta}A)_+<A$. Then $a_i-a_j-n e>0$.
\end{Lem}
\begin{proof}
We remark  that $\alpha_{(\sigma_\beta A)_+}=w \alpha_{\sigma_\beta A}$, where $w$ is some uniquely determined element in $\mathfrak{S}_{s_1}\times\ldots\times \mathfrak{S}_{s_\ell}$. The inequality $(\sigma_\beta A)_+<A$ just means that $\alpha_A-w \sigma_\beta \alpha_A$ is a combination of simple roots with non-negative integral coefficients (one of the coefficients should be strictly positive). We have $a_i-a_j-en=(\alpha_A,\beta)$.

Clearly,  $\alpha_A- w \sigma_\beta \alpha_A=(\alpha_A-\sigma_\beta \alpha_A)+(\sigma_\beta \alpha_A-w \sigma_\beta \alpha_A)$. The second summand is a combination of the simple roots of the Levi subalgebra $\gl_{s_1}\times\ldots\times \gl_{s_\ell}$. But $\beta$ does not lie in the span of those (otherwise $(\sigma_\beta A)_+=A$).
So if $(\sigma_\beta A)_+<A$, then the coefficient of $\beta$ in the first summand is positive, i.e.,
$a_i-a_j-en>0$.
\end{proof}

We are going to refine the ordering above. For this we will describe elements of $\Z^{\underline{s}}$ in a different
way -- as {\it virtual multipartitions}.

We will represent an element $A$ by an $\ell$-tuple of diagrams of a certain form of  that will be called {\it virtual Young diagrams}. Given a collection $\mu_1\geqslant\ldots\geqslant\mu_{s_i}$ of integers consider the shapes consisting of all unit squares with coordinates $(y,x)$ with $y\leqslant \mu_x$. Such a shape will be called a virtual
Young diagram. Unlike a usual Young diagram, a virtual one is infinite to the left but
still the rightmost positions of a box in a row increase from top to bottom.

Consider the element $A_\emptyset=(s_1,\ldots,1,s_2,\ldots,1,\ldots,s_{\ell},\ldots,1)$. We can view $A-A_\emptyset$
as a collection of $\ell$ virtual Young diagrams -- a {\it virtual  multi-partition}. So $\Z^{\bs}$ is in bijection with the set of all virtual multipartitions $(\mu^{(1)},\ldots,\mu^{(\ell)})$ such that $\mu^{(i)}$ consists precisely of $s_i$ rows.

Now we can introduce an ordering on $\Z^{\bs}$ similarly to the ordering on $\mathcal{P}_{\ell}$ from the previous subsection. Namely, given virtual multipartitions $\lambda,\mu\in {\Z}^{\underline{s}}$ we say that $\lambda\preceq \mu$
if we have orderings $(b_1,b_2,\ldots)$ and $(b_1',b_2',\ldots)$ of boxes in $\lambda$ and $\mu$,
respectively, such that $b_i\leqslant b_i'$ for every $i$.  We remark that there is an integer $k$ such that the parts of $\lambda,\mu$ lying to the left of the $k$th column (in all $\ell$ diagrams) coincide. So we actually have $b_i=b_i'$ for all $i$ but finitely many.

\begin{Lem}\label{Lem:ord_compat}
The partial order $\preceq$ refines  $\leq$. That is, $\lambda\leq \mu$ implies $\lambda\preceq\mu$.
\end{Lem}
\begin{proof}
We only need to prove that if $(\sigma_\beta A)_+<A$, then, for the corresponding virtual multipartitions
$\mu$ and $\lambda$, we have $\mu\prec\lambda$. We can assign an $\ell$-tuple of collections of boxes
to any element of $\Z^m$ similarly to what was done above. The element lies in $\Z^{\bs}$ if and only if
these shapes  satisfy the condition that the lengthes of the rows decrease from top to bottom.
We also can define a relation $\preceq$ as before on the set of these more general shapes but it will be a pre-order
instead of a partial order.

We will first describe the shape corresponding to $\sigma_\beta A$ and then explain how to get
the virtual multipartition corresponding to $(\sigma_\beta A)_+$ from that. We will see that
the shape $\lambda'$ corresponding to $\sigma_\beta A$ is $\preceq \lambda$, while $\mu$ and
$\lambda'$ are equivalent with respect to the preorder $\preceq$.

Let $\beta=\epsilon_i-\epsilon_j+ n\delta$ and $A=(a_1,\ldots,a_m)$. Then $\sigma_\beta A=(a_1',\ldots,a_m')$,
where $a_k'=a_k$ if $k\neq i,j$, and $a_i'=a_j+ ne, a_j'=a_i-ne$. So the virtual multipartition $\lambda'$
corresponding to $A'$ is obtained from $A$ by modifying two rows -- those corresponding to the indices $i,j$.
Namely, we modify the row corresponding to $i$ by removing $N:=a_i-a_j-ne$ boxes from there (recall that,
according to Lemma \ref{Lem:ordering}, $a_i-a_j-ne>0$). And we modify
the row corresponding to $j$ by adding $N$ boxes. Let us number the $N$ added boxes, $b'_1,\ldots,b'_N$, from left to right. Next, let us number the $N$ removed boxes $b_1,\ldots,b_N$, also from left to right. Then it is easy to check
that $b_i'\preceq b_i$ for all $i$.

Now let us explain how to transform $\lambda'$ to $\mu$. If in a virtual Young diagram $\nu$ we have $\nu_i<\nu_{i+1}$
then we take the last $\nu_{i+1}-\nu_i-1$ boxes in the $i+1$th row and move them to the $i$th row. We remark that
under this procedure each box remains in the same diagonal it has been. We apply this procedure as many times as possible.
 This is precisely the way to get $\mu$ from $\lambda'$ (if the shape that we get at the end is not a virtual
Young diagram, then one cannot transform $\sigma_\beta A$ into an element of $\Z^{\bs}$ by applying a permutation
from $\mathfrak{S}_{s_1}\times\ldots\times \mathfrak{S}_{s_\ell}$). The transformation does not change the equivalence
class of a shape with respect to the preorder $\preceq$. So we have proved that $\mu\preceq\lambda$.
\end{proof}

\subsection{Full parabolic affine category $\mathcal{O}$}
Let $m,e,{\underline{s}}=(s_1,\ldots,s_{\ell}),\g:=\gl_m, \hat{\g}, \h,\hat{\h}$ have the same meaning as in Section \ref{SS_parab_poset}.

Further let $\p$ be the parabolic
subalgebra in $\g$ that fixes the $\ell-1$ subspaces $\operatorname{Span}(e_1,\ldots,e_{s_1}),$ $ \operatorname{Span}(e_1,\ldots,e_{s_1+s_2}),
\ldots, \operatorname{Span}(e_1,\ldots,e_{s_1+\ldots+s_{\ell-1}})$. Its Levi subalgebra is isomorphic to $\gl_{s_1}\times\gl_{s_2}\times\ldots\times \gl_{s_\ell}$. Let $\OCat^{\p}_{-e}$ stand for the parabolic category
$\mathcal{O}$ for $\hat{\g}$ on level $-e$, whose objects are integrable over $\p\oplus t\g[t]$. See \cite[Section 2]{VV}
for details.

The category $\OCat^{\p}_{-e}$ together with parabolic Verma modules $\Delta(A), A\in \Z^{\bs},$ with $\rho$-shifted
highest weight $\alpha_A$ becomes a highest weight
category if we weaken (HW2) and allow the projective objects to lie in the pro-completion of this category.
For a highest weight order on $\Z^{\bs}$ we can take the order $<$ from Section \ref{SS_parab_poset},
see \cite[5.1,5.2]{VV}.

We remark that if $\Delta(A),\Delta(A')$ lie in the same block of $\OCat^{\p}_{-e}$, then  the $m$-tuples
of residues of $(a_1,\ldots,a_m)$ and of $(a_1',\ldots,a_m')$ modulo $e$ differ by a permutation.

Multiplicities (of the simples in the standards) for $\OCat^{\p}_{-e}$ are known, they are given by some
Kazhdan-Lusztig polynomials, see \cite[Proposition 5.8]{VV}.

Also we will need a deformation of $\OCat^{\p}_{-e}$. Let $\bar{\param}$ denote the space $\{(x_0,\ldots,x_\ell)| x_1+\ldots+x_{\ell}=0\}$, it is naturally identified with $\tilde{\param}$.
Consider the category $\OCat^{\p}_{-e,R}$ consisting of all $\hat{\g}\otimes R$-modules $M$ satisfying the following conditions:
\begin{itemize}
\item the action of $\hat{\g}\otimes R$ on $M$ is $R$-linear and $M$ is finitely generated over
$R\otimes U(\hat{\g})$.
\item the level of $M$ is $(x_0-\frac{1}{e})^{-1}$,
\item the action of $\p\oplus t\g[t]$ on $M$ is locally finite, meaning that every vector from $M$
lies in a finitely generated $\p\oplus t\g[t]$-stable $R$-submodule.
\item for any $i$, the element $\operatorname{id}_i\in \gl_{s_i}\subset \p$ acts on $M$ diagonalizably
with eigenvalues in $\Z+x_i$. Moreover, the action of $\gl_{s_i}$ on $M/(x_0,\ldots,x_\ell)$
integrates to $\operatorname{GL}_{s_i}$.
\end{itemize}

For example, we still have analogs $\Delta_{R}(A)$ of parabolic Verma modules in
$\OCat^{\p}_{-e,R}$.

\subsection{Truncated parabolic affine category $\mathcal{O}$}\label{SS_trunc_par}
Following \cite{VV}, we consider  certain truncations of $\OCat^{\p}_{-e}, \OCat^{\p}_{-e,R}$.
For a non-negative integer $n$ that is less then all $s_1,\ldots,s_{\ell}$ we are going to define a truncated subcategory $\OCat^{\p}_{-e}(\leqslant n)\subset \OCat^{\p}_{-e}$. Recall that in Section \ref{SS_parab_poset} we have identified the highest weight  poset $\Z^{\bs}$ of $\OCat^{\p}_{-e}$ with the set of virtual $\ell$-multipartitions $(\lambda^{(1)},\ldots,\lambda^{(\ell)})$,
where $\lambda^{(i)}$ has $s_i$ rows. We can embed the set $\mathcal{P}_\ell(n)$ of all $\ell$-multipartitions of $n$
into $\Z^{\bs}$ as follows. To a multipartition $\lambda=(\lambda^{(1)},\ldots,\lambda^{(\ell)})$ we
assign a virtual multipartition $\tilde{\lambda}=(\tilde{\lambda}^{(1)},\ldots,\tilde{\lambda}^{(\ell)})$ in the following way: the columns of $\tilde{\lambda}^{(i)}$ with positive numbers are the same as of $\lambda^{(i)}$, while the columns with non-positive numbers consist of precisely $s_i$ elements. Since $n< s_i$ for all $i=1,\ldots,\ell$, this map is well defined. Below we will always view $\mathcal{P}_\ell(n)$ as a subset of $\Z^{\bs}$ in this way.
Set $\mathcal{P}_\ell(\leqslant n)=\bigsqcup_{j\leqslant n}\mathcal{P}_\ell(j)$.
By \cite[Proposition A6.1]{VV},   if $\lambda\in \mathcal{P}_\ell(\leqslant n)$
and $\mu\in \Z^{\bs}$ is less than $\lambda$ (in the linkage order recalled in Section \ref{SS_parab_poset}), then $\mu\in \mathcal{P}_\ell(\leqslant n)$ too.

Consider the Serre subcategory $\OCat^{\p}_{-e}(\leqslant n)$ of $\OCat^\p_{-e}$ generated by $\Delta(\lambda), \lambda\in \mathcal{P}_\ell(\leqslant n)$.
From the result quoted in the end of the previous paragraph, it follows
that $\OCat^{\p}_{-e}(\leqslant n)$ is a highest weight subcategory of $\OCat^{\p}_{-e}$ with standard
objects $\Delta(\lambda), \lambda\in \mathcal{P}_\ell(\leqslant n)$. We have a natural direct sum decomposition
$\OCat^{\p}_{-e}(\leqslant n)=\bigoplus_{j=0}^{n}\OCat^{\p}_{-e}(j)$. A deformation $\OCat^{\p}_{-e,R}(\leqslant n)$
of $\OCat^\p_{-e}(\leqslant n)$ is constructed via a similar truncation starting from $\OCat^{\p}_{-e,R}$. It is not difficult to see that it is equivalent
to the category of modules over an associative algebra that is free of finite rank over $R$.
It is a highest weight category over $R$ in the sense of Rouquier, \cite[4.1]{rouqqsch}.

\begin{Conj}[\cite{VV}]\label{Conj:VV}
There is an equivalence $\OCat^{\p}_{-e}(n)\xrightarrow{\sim}\OCat_{\kappa,{\underline{s}}}(n)$ that maps $\Delta(\lambda)\in \OCat^{\p}_{-e}(n)$ to $\Delta(\lambda)\in \OCat_{\kappa,{\underline{s}}}(n)$.
\end{Conj}

We will prove this conjecture when $m\gg 0$ (we will explain how big $m$ one has to take later).
The multiplicities in $\OCat^{\p}_{-e}(n)$ are known, see \cite[Section 8]{VV}:  under the usual identification of
$[\OCat^{\p}_{-e}(j)]$ with the degree $j$ part of the level $\ell$ Fock space  $\mathcal{F}^{\bs}$,
the classes of simples are the elements of Uglov's dual canonical basis
introduced in \cite{Uglov}. This gives multiplicities in
$\OCat_{\kappa,{\underline{s}}}(n)$. As indicated in \cite{VV}, these multiplicity formulas are exactly as predicted
by a conjecture of Rouquier, \cite[6.5]{rouqqsch}.

\begin{Rem}
An assumption that $m$ is very large is crucial for some of techniques we use
(combinatorial sufficient conditions for faithfulness to be discussed
in Section \ref{S_faith_comb}). We do not have a proof without this assumption.
Conjecture \ref{Conj:VV} under the assumption $n\leqslant s_i$ for all $i$
is proved in \cite[Theorem 6.9]{RSVV}.
\end{Rem}

\section{Faithfulness}\label{S_faith}
Here we are going to provide an extension of a technique used by Rouquier in \cite{rouqqsch} to prove an analog
of Conjecture  \ref{Conj:VV} in the $\ell=1$ case. Rouquier's technique requires to check that certain functors
are faithful on standardly filtered objects.

\subsection{Checking faithfulness}
Let $R$ be an algebra of formal power series  over $\C$, $R=\C[[V]]$, where $\dim V=2$,
$p$ be the maximal ideal,
and $\OCat_R$ be a highest weight category over $R$ with labeling set $\Lambda$. When we say this, we mean, in particular,
that $\OCat_R$ is the category of modules over a free finite rank algebra $A_R$ over $R$. We remark that all standardly filtered modules are free over $R$.

Choose some projective object, say $P_p$, in the specialization $\OCat_{p}$ and extend it to
a projective $P_R$ in $\OCat_R$. Let $\Cat_R$ be the quotient category associated to $P_R$. On the level
of $A_R$-modules, the quotient functor is just the multiplication by an idempotent.

We say that the quotient morphism $\pi_p$ is $(-1)$-faithful
if it is faithful on standardly filtered objects, i.e., objects
that admit a filtration with subsequent quotients $\Delta_p(\lambda),\lambda\in \Lambda$.
This is equivalent to say that a simple not covered by $P_p$ cannot appear in the socle of a standard object.
We say that $\pi_p$ is $0$-faithful if it is fully faithful on standardly filtered
objects. Finally, we say that $\pi_p$ is $1$-faithful if it is $0$-faithful
and, in addition, $\operatorname{Ext}^1_{\OCat_p}(M,M')=\Ext^1_{\Cat_p}(\pi(M),\pi(M'))$
for any standardly filtered objects $M,M'$. Of course, we can give completely analogous
definitions for $\OCat_R$ or for any specialization of $\OCat_R$.

We make  the following assumption  on $\OCat_R$:
\begin{itemize}
\item[($\heartsuit$)] There are $y_1,\ldots,y_k\in V^*$  such that after localizing $y=y_1\ldots y_k$ the category
$\OCat_R$ becomes equivalent to the category of modules over the direct sum of matrix algebras and  the functor $\pi_R$ becomes
an equivalence of the categories.
Further, for any point $\p$ of codimension $1$, the functor $\pi_{\p}$ is $(-1)$-faithful.
\end{itemize}

Of course, we can assume that $y_1,\ldots,y_k$ are pairwise non-proportional.

The following proposition  is an elaboration of Rouquier's results, \cite[4.2]{rouqqsch}.

\begin{Prop}\label{Prop:faith}
Suppose that  ($\heartsuit$) holds. Then the functor $\pi_R$ is $0$-faithful.
\end{Prop}

\begin{Lem}\label{Lem:Hom_freeness}
Let $A_R$ be an associative algebra that is free of finite rank over $R$.
Let $M_R,M'_R$ be $A_R$-modules that are free over $R$. Then $\Hom_{A_R}(M,M')$ is a free $R$-module.
\end{Lem}
\begin{proof}
Pick a basis $x_1,x_2\in V^*$. Then we have an exact sequence
$$0\rightarrow \Hom_{A_R}(M,M')\xrightarrow{x_1} \Hom_{A_R}(M,M')
\rightarrow \Hom_{A_R/(x_1)}(M/x_1M,M'/x_1M').$$
The last term is a torsion free and hence free $R$-module.
We conclude that $x_1,x_2$ is a regular sequence for $\Hom_{A_R}(M,M')$.
Therefore $\Hom_{A_R}(M,M')$ is a maximal Cohen-Macaulay $R$-module.
Since $R=\mathbb{C}[[x_1,x_2]]$, we conclude that $\Hom_{A_R}(M,M')$ is free.
\end{proof}

\begin{proof}[Proof of Proposition \ref{Prop:faith}]
%
Set $R_1:=R/(y^s)$, where $s$ is a positive integer.
Choose standardly filtered objects $M,M'$ in $\OCat_R$ and let $M_1,M_1'$ be their specializations to $\OCat_{R_1}$.
We have the following commutative diagram, where $\Hom$'s are taken in $\OCat_R,\Cat_R$.

\begin{center}
\begin{picture}(100,50)
\put(4,2){$\Hom(M_1,M'_1)$}
\put(50,2){$\Hom(\pi_{R_1}M_1, \pi_{R_1}M_1)$}
\put(5,15){$\Hom(M,M')$}\put(50,15){$\Hom(\pi_{R} M, \pi_{R} M')$}
\put(5,28){$\Hom(M,M')$}\put(50,28){$\Hom(\pi_{R} M, \pi_{R} M)$}
\put(14,41){$0$}\put(60,41){$0$}
\put(15,40){\vector(0,-1){8}}
\put(15,27){\vector(0,-1){8}}
\put(15,14){\vector(0,-1){8}}
\put(61,40){\vector(0,-1){8}}
\put(61,27){\vector(0,-1){8}}
\put(61,14){\vector(0,-1){8}}
\put(30,16){\vector(1,0){19}}
\put(30,29){\vector(1,0){19}}
\put(30,3){\vector(1,0){19}}
\put(16,22){{\footnotesize $y^s$}}
\put(62,22){{\footnotesize $y^s$}}
\end{picture}
\end{center}

Note that since $M,M'$ are free over $R$, the $R$-module $\Hom(M,M')$
is free, see Lemma \ref{Lem:Hom_freeness}.
The same is true for $\Hom(\pi_R M,\pi_R M')$. So $\Hom(M,M')/y^s \Hom(M,M')$ is free over $R/(y^s)$.
Since the functor
$\pi_{\p}$ is $(-1)$-faithful in the generic point $\p$ of each
$R/(y_i)$ we see that the kernel of the bottom horizontal arrow is
supported at $p$. It follows that its intersection with the
submodule  $\Hom(M,M')/y^s \Hom(M,M')\subset \Hom(M_1,M_1')$
is zero. We deduce that $\Hom(M,M')/y^s\Hom(M,M')
\hookrightarrow \Hom(\pi_R M,\pi_R M')/y^s\Hom(\pi_R M,\pi_R M')$.
By the choice of $y$, the kernel and the cokernel of $\Hom(M,M')
\rightarrow \Hom(\pi_R M, \pi_R M')$ are annihilated by some power
of $y$, we may assume that this power is $y^s$. Note that this implies
that the ranks of the free $R$-modules  $\Hom(M,M')$,
$\Hom(\pi_R M,\pi_R M')$ coincide and so the homomorphism
$\Hom(M,M')\rightarrow \Hom(\pi_R M,\pi_R M')$ is given by multiplication
with a square matrix, say $A$. Since
$$\Hom(M,M')/y^s \Hom(M,M')\hookrightarrow
\Hom(\pi_R M,\pi_R M')/y^s \Hom(\pi_R M,\pi_R M'),$$
we see that $\det(A)$ is invertible at the generic point of every $R/(y_i)$.
Also it is invertible after localizing $y$. We deduce that $\det(A)$
is invertible in $R$ and this completes the proof.
\end{proof}

The following lemma is a direct corollary of \cite[Proposition 4.42]{rouqqsch}.

\begin{Lem}\label{Lem:1_faith}
If $\pi_p$ is $0$-faithful, then $\pi_R$ is 1-faithful.
\end{Lem}

Note that all results in this section are still true when $R$ is a formal power series ring in
one variable.

\subsection{Category equivalence}
Suppose that $R$ is a formal power series ring over $\C$.
Let $\OCat^1_R,\OCat^2_R$ be two highest weight categories. Pick projectives
$\bar{P}^i_R\in \OCat^i_R$ with $\End_{\OCat^1_R}(\bar{P}^1_R)\cong
\End_{\OCat^2_R}(\bar{P}^2_R)$. Set $\Cat_R:=\End_{\OCat^i_R}(\bar{P}^i_R)^{opp}\operatorname{-mod}$.
Let $\bar{\pi}^i_R:\OCat^i_R\rightarrow \Cat_R$ denote the quotient functor
defined by $\bar{P}^i_R, i=1,2$.
Let $P^i_R$ be a summand of $\bar{P}^i_R$ that is  injective after specialization to the closed point 
$p$ and generates  $\OCat^i_R$
after passing to the generic point of $R$. Suppose that $\bar{\pi}^1_R(P^1_R)=\bar{\pi}^2_R(P^2_R)$.
As in \cite[Section 4.2]{rouqqsch}, this gives rise to the identification of the sets of
irreducibles in $\OCat^1_p,\OCat^2_p$.

\begin{Thm}\label{Thm:Thm_equi}
In the notation above,  suppose that the following holds.
\begin{itemize}
\item[(i)] There is a common highest weight order on the identified sets
$\operatorname{Irr}(\OCat^i_p)$.
\item[(ii)] The functor  $\overline{\pi}^2_R$ is 1-faithful.
\item[(iii)] The functor $\overline{\pi}^1_R$ is 0-faithful.
\item[(iv)] Every   projective in $\OCat^2_R$ admits an
inclusion into a  $(P^2_R)^{\oplus m}$ (for some $m$)
with standardly filtered cokernel. 
\end{itemize}
Then there is an equivalence $\OCat^1_R\xrightarrow{\sim}
\OCat^2_R$ intertwining the functors $\overline{\pi}^1_R,\overline{\pi}^2_R$.
This equivalence restricts to an equivalence $\OCat^{1\Delta}_R\xrightarrow{\sim}
\OCat^{2\Delta}_R$ and the restriction coincides with $(\overline{\pi}^2_R)^*\circ \overline{\pi}^1_R$.
\end{Thm}

We will give a proof after two auxiliary lemmas. The first one is standard.

\begin{Lem}\label{Lem:im_proj}
Let $\OCat^1_R,\OCat^2_R$ be two highest weight categories over $R$
with the same labeling sets and highest weight orders.
Suppose that there is a fully faithful inclusion $\iota:\OCat^{1\Delta}_R\hookrightarrow \OCat^{2\Delta}_R$
that maps $\Delta^1_R(\lambda)$ to $\Delta^2_R(\lambda)$ for all $\lambda$.
This embedding is an equivalence if and only if the image contains $\OCat^2_R\operatorname{-proj}$.
In this case, we have an equivalence $\OCat^1_R\xrightarrow{\sim}\OCat^2_R$
extending $\iota$.
\end{Lem}
\begin{proof}
We need to show that
if $\OCat^2_R\operatorname{-proj}\subset \iota(\OCat^{1\Delta}_R)$, then
$\iota(\OCat^{1}_R\operatorname{-proj})=\OCat^{2}_R\operatorname{-proj}$.
By \cite[Lemma 4.22]{rouqqsch},
$\OCat^2_R\operatorname{-proj}\subset \iota(\OCat^1_R\operatorname{-proj})$.
Since the number of the indecomposable projectives in $\OCat^1_R,\OCat^2_R$
coincide, we see that   $\OCat^2_R\operatorname{-proj}= \iota(\OCat^1_R\operatorname{-proj})$.
\end{proof}

\begin{Lem}\label{Lem:R1_vanish}
Under the assumptions of Theorem \ref{Thm:Thm_equi},
let $M\in \OCat^{1\Delta}_R$ be such that
$M$ is included into $(P^1_R)^{\oplus m}$ with standardly filtered cokernel.
Then $R^1(\bar{\pi}_R^1)^*(M)=0$.
\end{Lem}
\begin{proof}
Let $N$ denote the cokernel of the inclusion $M\hookrightarrow (P^1_R)^{\oplus m}$
in the lemma.
We have a long exact sequence
\begin{align*}0\rightarrow &(\bar{\pi}^1_R)^*\circ \bar{\pi}^1_R(M)\rightarrow
(\bar{\pi}^1_R)^*\circ \bar{\pi}^1_R(P^1_R)^{\oplus m}\rightarrow
(\bar{\pi}^1_R)^*\circ \bar{\pi}^1_R(N)\rightarrow\\ &R^1(\bar{\pi}^1_R)^*\circ \bar{\pi}^1_R(M)\rightarrow
R^1(\bar{\pi}^1_R)^*\circ \bar{\pi}^1_R(P^1_R)^{\oplus m}.\end{align*}
Since $\bar{\pi}^1_R$ is $0$-faithful, we see that the first line is
just $0\rightarrow M\rightarrow (P^1_R)^{\oplus m}
\rightarrow N$. So $R^1(\bar{\pi}^1_R)^*\circ \bar{\pi}^1_R(M)\hookrightarrow
R^1(\bar{\pi}^1_R)^*\circ \bar{\pi}^1_R(P^1_R)^{\oplus m}$. But $\bar{\pi}^1_p(P^1_p)$
is injective in $\Cat_p$ and so the target object is 0. This completes the proof of this lemma.
\end{proof}

\begin{proof}[Proof of Theorem \ref{Thm:Thm_equi}]
Note that $\iota:=(\pi^2_R)^*\circ\pi^1_R:\OCat^{1\Delta}_R\rightarrow
\OCat^{2\Delta}_R$ is a fully faithful embedding mapping $\Delta^1_R(\lambda)$
to $\Delta^2_R(\lambda)$, see the proofs of \cite[Lemma 4.48,Theorem 4.49]{rouqqsch}.

Let us order the labels $\lambda_1,\ldots,\lambda_r$ for $\OCat^i_R$ in such
a way that $\lambda_i>\lambda_j$ implies $i<j$. Let us write
$(\OCat^i_R)^{\Delta,\leqslant k}$ for the full subcategory
of $\OCat^{i\Delta}_R$ consisting of all objects filtered
by $\Delta^i_R(\lambda_j)$ with $j\leqslant k$. For $M\in \OCat^{i\Delta}_R$, let us write
$M_{\leqslant j}$ for the uniquely defined subobject of
$M$ that is filtered by $\Delta^i_R(\lambda_\ell)$ with $\ell\leqslant j$,
while the quotient $M/M_{\leqslant j}$ is filtered
with $\Delta^i_R(\lambda_\ell)$ with $\ell>j$.

We will prove that
\begin{itemize}\item[$(E_k)$] $\iota\left((\OCat^1_R)^{\Delta,\leqslant k}\right)=(\OCat^2_R)^{\Delta,\leqslant k}$
\end{itemize}
using an induction on $k$. Note that $(E_r)$ implies the claim of the theorem thanks to
Lemma \ref{Lem:im_proj}. The base $k=1$ follows from the claim that
$\Delta^2_R(\lambda)=\iota(\Delta^1_R(\lambda))$
for any $\lambda$. The proof of the induction step, $(E_k)\Rightarrow (E_{k+1})$
is in two substeps: we first show that $(E_k)$ implies
\begin{itemize}
\item[$(I_k)$] For any $\hat{P}^2_R\in \OCat^2_{R}\operatorname{-proj}$,
there is an inclusion $\hat{P}^2_{R,\leqslant k}$ into $\iota(P^1_R)^{\oplus m}$
with standardly filtered cokernel.
\end{itemize}
Then we show that $(I_k)\Rightarrow (E_{k+1})$.


{\it Step 1}. Let $M$ be such as in Lemma \ref{Lem:R1_vanish} and $N$ be any object in $\OCat^{1\Delta}_R$.
We claim that if $Q\in \OCat^{2\Delta}_R$ is included into an exact sequence
$0\rightarrow \iota(M)\rightarrow Q\rightarrow \iota(N)\rightarrow 0$, then $Q\in \iota(\OCat^{1\Delta}_R)$.
Indeed, we have an exact sequence $0\rightarrow M\rightarrow (\bar{\pi}^1_R)^*\circ
\bar{\pi}^2_R(Q)\rightarrow N\rightarrow R^1(\bar{\pi}^1_R)^*\circ \bar{\pi}^1_R(M)$ in $\OCat^1_R$.
From this sequence and Lemma \ref{Lem:R1_vanish}, we see that $(\bar{\pi}^1_R)^*\circ
\bar{\pi}^2_R(Q)\in \OCat^{1\Delta}_R$. This implies $Q\in \iota(\OCat^{1\Delta}_R)$.

{\it Step 2}. Now we prove $(E_k)\Rightarrow (I_k)$.

 Note that, for $M,M'\in \OCat^{i\Delta}_R$,
an inclusion $M\hookrightarrow M'$ with standardly filtered
cokernel restricts to an inclusion $M_{\leqslant j}\hookrightarrow
M'_{\leqslant j}$ with standardly filtered cokernel.

Let $\hat{P}^2_R\in \OCat^2_R\operatorname{-proj}$. We claim
that $\iota^{-1}(\hat{P}^2_{R,\leqslant k})\in \OCat^{1\Delta}_R$ is included
into $(P^1_R)^{\oplus m}$ with standardly filtered cokernel, this claim is equivalent
to $(I_k)$.  Indeed, from the inclusion $\hat{P}^2_{R}\subset (P^2_R)^{\oplus m}$
with standardly filtered cokernel, we deduce the inclusion  $\hat{P}^2_{R,\leqslant k}
\subset(P^2_{R,\leqslant k})^{\oplus m}$ with standardly
filtered cokernel. The argument of the proof of \cite[Lemma 4.48]{rouqqsch}
shows that $\bar{\pi}^1_R(P^1_{R,\leqslant k})=
\bar{\pi}^2_R(P^2_{R,\leqslant k})$ in $\Cat_R$.
We conclude that $\iota(P^1_{R,\leqslant k})=
P^2_{R,\leqslant k}$. So  $\hat{P}^2_{R,\leqslant k}\hookrightarrow
\iota(P^1_R)^{\oplus m}_{\leqslant k}$ with standardly filtered cokernel. This implies $(I_k)$.

{\it Step 3}. We now proceed to proving $(I_k)\Rightarrow (E_{k+1})$. Let
$\hat{P}^2_R\in \OCat^2_R\operatorname{-proj}$.
Let us prove, first, that $\hat{P}^2_{R,\leqslant k+1}\in \iota(\OCat^{1\Delta}_R)$.
We have an exact sequence
$0\rightarrow \hat{P}^2_{R,\leqslant k}\rightarrow \hat{P}^2_{R,\leqslant k+1}\rightarrow
\Delta^2_R(\lambda_{k+1})^{\oplus s}\rightarrow 0$ for some $s\geqslant 0$. By Step 2, we can take
$M:=\iota^{-1}(\hat{P}^2_{R,\leqslant k})$ in Step 1. The inclusion
$\hat{P}^2_{R,\leqslant k+1}\in \iota(\OCat^{1\Delta}_R)$ follows now from Step 1.

{\it Step 4}. In particular, for $j\leqslant k+1$, the objects $P_R(\lambda_j)$  lie in
$\iota\left((\OCat^1_R)^{\Delta,\leqslant k+1}\right)$. By Lemma \ref{Lem:im_proj},
$\iota\left((\OCat^1_R)^{\Delta,\leqslant k+1}\right)=(\OCat^2_R)^{\Delta,\leqslant k+1}$,
which is $(E_{k+1})$.
\end{proof}

\subsection{Example: KZ functor}\label{SS_KZ}
Let us provide an example of a quotient functor that will be of importance for us: the KZ functor introduced
in \cite{GGOR}.

First, we need to recall the definition of a cyclotomic Hecke algebra.
Consider the affine Hecke algebra $\Hecke^{aff}_q(j)$ that is generated by elements $T_1,\ldots,T_{j-1},X_1$ subject to
the relations \begin{align*}&(T_i+1)(T_i-q^2)=0, T_i T_{i+1}T_i=T_{i+1}T_iT_{i+1}, T_i T_j=T_j T_i, |i-j|>1,\\ &T_1 X_1 T_1 X_1=X_1 T_1 X_1 T_1, X_1 T_i=T_i X_1, i>1,\end{align*} where, as before, $q=\exp(\pi \sqrt{-1}/e)$.
We will consider the cyclotomic quotient $\Hecke^{\bs}_q(j)$ of $\Hecke^{aff}_q(j)$ by $\prod_{j=1}^\ell (X_1-q^{2s_i})$.
Also we consider the  deformation $\Hecke^{\bs}_{q,R}(j)$ of $\Hecke^{\bs}_q(j)$ that is an
algebra over $R=\C[\param]^{\wedge_p}$.
Namely, we impose relations $(T_i+1)(T_i-\tilde{q}^2), \prod_{j=1}^\ell (X_1-\tilde{q}^{2\tilde{s}_i})$,
where $\tilde{q}=\exp(\pi\sqrt{-1}(\frac{1}{e}+ x_0)), \tilde{s_i}=s_i+ x_i, i=1,\ldots,\ell$ (where we normalize
$x_1,\ldots,x_\ell$ by $x_1+\ldots+x_\ell=0$) and leave all other relations unchanged.
It is known that the localization $\operatorname{Frac}(R)\otimes_{R}\Hecke_{q,R}^{\bs}(j)$
is a split semisimple algebra over $\operatorname{Frac}(R)$, see \cite{ArikiKoike}.
We also can consider the deformation
$\Hecke^{aff}_{q,R}(j)$ of $\Hecke^{aff}_q(j)$.

Then according to \cite{GGOR}, there is a projective object $P_{KZ}$ in $\OCat_{\kappa,\bs}(j)$ that defines
the quotient functor $\operatorname{KZ}:\OCat_{\kappa,\bs}(j)\twoheadrightarrow \Hecke^{\bs}_q(j)\operatorname{-mod}$. Moreover,
the deformation $P_{KZ,R}$ of $P_{KZ}$ to $\OCat_{\kappa,\bs,R}(j)$ defines the quotient functor
$\operatorname{KZ}_{\param}:\OCat_{\kappa,\bs,R}(j)\twoheadrightarrow \Hecke^{\bs}_{q,R}(j)\operatorname{-mod}$.

Let us describe the faithfulness properties of $\operatorname{KZ}$ established in \cite[Section 5]{GGOR}.
First of all, $\operatorname{KZ}$ is $(-1)$-faithful,
this is clear from the construction. Also it is fully faithful on projectives.
Finally, under certain conditions on $p$, the functor $\operatorname{KZ}$ is $0$-faithful.
Namely, this is equivalent to $e>2$ and $s_i-s_j$ not divisible by $e$ for any different $i,j$.

\begin{Prop}\label{Prop:Cher_proj_emb}
Let $P_R$ be a projective object in $\mathcal{O}_{\kappa,\bs,R}(n)$ and let $P_{KZ,R}$
denote the projective object defining the KZ functor
$\mathcal{O}_{\kappa,\bs,R}(n)\rightarrow \Hecke^{\bs}_{q,R}(n)\operatorname{-mod}$.
Then there is $m>0$ and an embedding
$P_R\hookrightarrow P_{KZ,R}^{\oplus m}$ with standardly filtered cokernel.
\end{Prop}
\begin{proof}
The proof follows that of \cite[Theorem 5.1]{LW}.
Let $\operatorname{Res}: \OCat_{\kappa,\bs,R}(n)\rightarrow R\operatorname{-mod}$ be the Bezrukavnikov-Etingof restriction functor corresponding to
the parabolic subgroup $\{1\}\subset G_n$ and let $\operatorname{Ind}:R\operatorname{-mod}
\rightarrow \OCat_{\kappa,\bs,R}(n)$ be the corresponding induction functor, see \cite[Section 3.5]{BE},
a two-sided adjoint to $\Res$.
We have $P_{KZ,R}=\operatorname{Ind}(R)$, \cite[Example 3.15]{BE}. So we get an adjointness
homomorphism $P_R\rightarrow \Ind(\Res(P_R))$. It is an embedding after the specialization
at $p$ because the socle of $P_p$ does not contain simples
annihilated by $\Res$. It follows that the cokernel $C_R$ of $P_R\rightarrow
\Ind(\Res(P_R))$ is free over $R$. Let us check that $C_R$ is standardly
filtered. By \cite[Lemma 4.21]{rouqqsch}, it is enough to check that $\operatorname{Ext}^i(C_R,\nabla_R(\lambda))$
equals zero for all $\lambda$ and all $i>0$. For $i>1$ this is automatic as $C_R$ has a projective resolution
of length $1$. To check that $\operatorname{Ext}^1(C_R,\nabla_R(\lambda))=0$ it is enough to show
that $\operatorname{Hom}(\Ind(\Res(P_R)), \nabla_R(\lambda))\twoheadrightarrow
\operatorname{Hom}(P_R,\nabla_R(\lambda))$. The first term is naturally identified
with  $\operatorname{Hom}(P_R, \Ind(\Res(\nabla_R(\lambda))))$. So it is enough to
check that $\Ind(\Res(\nabla_R(\lambda)))\twoheadrightarrow \nabla_R(\lambda)$.
The latter is a consequence of the fact that the head of $\nabla_p(\lambda)$ does not
contain simples annihilated by $\Res$, which in its turn follows from the analogous
claim for the socles of standard objects, because costandards are obtained from
standards by duality, see \cite[Section 4.2]{GGOR}.
\end{proof}

\section{Categorical Kac-Moody actions}\label{S_cat_KM}
According to \cite{Rouquier_2Kac}, one gets a
functor from $\OCat_R$ to $\bigoplus_{j\geqslant 0}\Hecke_{q,R}^{\bs}(j)\operatorname{-mod}$
if $\OCat_R$ is equipped with an appropriate type A
categorical Kac-Moody action. So in this section we recall the categorical Kac-Moody actions on
$\OCat_{\kappa,{\underline{s}}}$ (and its deformation) and on $\OCat^{\p}_{-e}$ (and its deformation).
Further, we will recall some results of \cite{cryst},\cite{str} on highest weight categorical
$\sl_2$-actions. These results will be used in the next section to produce a "restricted"
categorical Kac-Moody action on $\OCat^{\p}_{-e, R}(\leqslant n)$. Also they will
be used several times below, in particular, to relate the faithfulness properties of quotient
functors to some purely combinatorial questions.

\subsection{Definition and examples}\label{SS_categorif1}
First, let us recall the notion of a type A categorical action (=categorification), essentially due to Chuang and Rouquier,
\cite{CR}, see also \cite[5.3.8]{Rouquier_2Kac}.

Let $\Cat$ be an artinian abelian $\C$-linear category. Then by a type A categorification on $\Cat$
one means a pair of biadjoint functors $E,F$ with fixed one-sided  adjointness morphisms
$\operatorname{Id}\rightarrow EF, FE\rightarrow \operatorname{Id}$ and also functor morphisms
$X\in \operatorname{End}(F), T\in \operatorname{End}(F^2)$ subject to the following condition:
the assignment $X_i\mapsto 1_F^{n-i}X 1_F^{i-1}, T_j\mapsto 1_F^{n-j-1}T 1_F^{j-1}$ extends
to an algebra homomorphism $\Hecke^{aff}_q(n)\rightarrow \End(F^n)$ for any $n$
(here $1_F$ stands for the identity transformation of $E$ and $q$ is a non-zero element of $\C$).
Similarly, we can define a type A categorification on an $R$-linear category
($q$ is now required to be an invertible element of $R$), here we are required to
get a homomorphism $\Hecke^{aff}_{q,R}(n)\rightarrow \End(F^n)$.

We proceed by recalling type A categorifications on the categories $\OCat_{\kappa,\bs,R}$
and $\OCat^{\p}_{-e,R}$.

A categorification on $\OCat_{\kappa,\bs,R}$ was introduced by Shan, \cite{Shan} (in the specialized
setting, the deformed setting is completely analogous).
The functors $F,E$ on the direct summand $\OCat_{\kappa, \bs,R}(n)$ are the Bezrukavnikov-Etingof induction
(from $G_n$ to $G_{n+1}$) and restriction  (from $G_n$ to $G_{n-1}$) functors, \cite{BE}. The transformations
$X$ and $T$ are defined using the KZ functor recalled in
Section \ref{SS_KZ}. The reader is referred to \cite[Section 5]{Shan}
for details. In particular, this description shows that the object $F^n \Delta_R(\varnothing)$ is projective
and realizes the KZ functor $\operatorname{KZ}: \OCat_{\kappa, \bs, R}(n)\rightarrow\operatorname{mod}$-
$\Hecke_{q,R}^{\bs}(n)$ as
$\Hom(F^n\Delta_R(\varnothing),\bullet)$.

Now let us explain how to equip $\OCat^{\p}_{-e,R}$ with a type $A$ categorification
that essentially appeared in \cite{VV}.

Namely, consider the Kazhdan-Lusztig category $\OCat^{\g}_{-e,R}$ over $R$ (that is the special case of
the category $\OCat^{\p}_{-e,R}$ for $\p=\g$) and the specialization $\OCat^{\g}_{-e}$ of $\OCat^{\g}_{-e,R}$ at the maximal
ideal (of course, the deformation with respect to $x_1,\ldots,x_\ell$ is trivial).
For a finitely dimensional $\operatorname{GL}_m$-module $L$ let $M(L),M_R(L)$ denote the corresponding
Weyl modules in $\OCat^{\g}_{-e}$ and $\OCat^{\g}_{-e,R}$, of course, $M(L)$
is a specialization of $M_R(L)$.  Also recall the duality functor $D$, see e.g. \cite[Section 2.6]{VV}, on $\OCat_{-e,R}^{\g}$.
It sends $M_R(L)$ to $M_R(L^*)$.

For two modules in
$\OCat^{\g}_{-e,R}$ we can take their Kazhdan-Lusztig tensor product (also known as the fusion product)
$\dot{\otimes}$, see \cite[Section 4]{KL} for the definition over $\C$ and \cite[Section 8]{KL}
for the extension to $R$ (see also \cite[Appendix 2.1]{VV}).
To define this tensor product we need to fix some points in $\mathbb{P}^1$ but the result
is independent of this choice up to an isomorphism. The bifunctor $\dotimes$  turns $\OCat^{\g}_{-e,R}$ into
a braided monoidal category, whose unit object is $\tilde{M}(\C)$,
where $\C$ is the trivial $\operatorname{GL}_n$-module, see \cite[Section 31]{KL}.

Also for $M\in \OCat^{\p}_{-e,R}$
and $N\in \OCat^{\g}_{-e,R}$ we can define their product $ N\dotimes M$ (whose definition again depends
on a choice of points but which is well-defined up to an isomorphism). The $\hat{\g}\otimes R$-module
$N\dotimes M$ is still in $\OCat^{\p}_{-e,R}$. This defines a right biexact bifunctor
$\OCat^{\g}_{-e,R}\times \OCat^{\p}_{-e,R}\rightarrow \OCat^{\p}_{-e,R}$ that turns $\OCat^{\p}_{-e,R}$
into a module category over $\OCat^{\g}_{-e,R}$, \cite[Proposition 7.2]{VV}.

Now suppose that $N$ is a standardly filtered object in $\OCat^{\g}_{-e,R}$. Then the endofunctor $N\dotimes\bullet$ of
$\OCat^{\p}_{-e,R}$ is exact, and biadjoint to $D(N)\dotimes \bullet$, see \cite[Corollary 7.3]{VV}.

We will be interested in the case when $N=M_R(\C^m)$. We denote the functor $M_R(\C^m)\dotimes \bullet$ by $F$
and the biadjoint functor $M_R(\C^{m*})\dotimes\bullet$ by $E$. There are distinguished elements $X\in \End(F)$ and
$T\in \End(F^2)$. They are given by monodromy. Namely, recall that to define $F$ we need to fix 3 points:
say $0,z,\infty$, where we put the modules $\tilde{M}\in \OCat^{\p}_{-e,R}, M_R(\C^m)$ and the resulting module
$M_R(\C^m)\dotimes_z M$, respectively, where the subscript ``$z$'' indicates the
dependence on $z$. Then we get a local system over $\C^\times$ with fiber $M_R(\C^m)\dotimes_z M$ over $z$.
The corresponding connection is the $n=1$ special case of the KZ connection: $\frac{d}{dz}-\frac{\Omega}{z}$,
where $\Omega$ is the tensor Casimir, $\Omega=\sum_{i,j=1}^n E_{ij}\otimes E_{ji}$, where $E_{ij}$ denote
the matrix unit on the position $(i,j)$. See, for example, \cite[Lecture 3]{EFK}. The monodromy
of this local system around 0 gives us the transformation $X$ we want. Similarly, to define $T\in \End(F^2)$, we need
to choose two points $z_1,z_2\in \C^\times$. The transformation $T$ is defined similarly to $X$ using the monodromy
of the path interchanging $z_1$ and $z_2$ (in fact, we need to multiply this transformation $T$ by $\tilde{q}$).

Now fix $j\geqslant 1$. By the construction, we get an $\mathfrak{S}_j$-equivariant local system on $(\C^\times)^{reg}:=\{(z_1,\ldots,z_j)\in \C^\times| z_k\neq z_\ell\}$ with fibers $F^j M$,
compare to \cite[Section 2.19]{VV}. The elements $T_0$ given by the monodromy of $z_1$ going around $0$ and $T_i$ given by interchanging $z_{i},z_{i+1}$ define a representation
of the affine braid group $B^{aff}(j)$ on $F^j M$ that is functorial in $M$. By the associativity of
the Kazhdan-Lusztig tensor product, the operators $T_1,\ldots,T_{j-1}$ satisfy the relations
of the analogously defined operators on $N^{\dotimes j}$. By \cite[Section 35]{KL}, the braided
monoidal categories $(\OCat^{\g}_{-e,R},\dotimes)$ and $(U_{q,R}(\g),\otimes)$ are equivalent
(here $U_{q,R}(\g)$ is Lusztig's form of the quantized universal enveloping algebra of $\g$
over $R$),
and the equivalence sends $N$ to the tautological module $R^m$. In particular, each
$T_i$ satisfies the Hecke relation $(T_i-\tilde{q}^2)(T_i+1)=0$. Our conclusion is that
$F^n M$ becomes a module over   $\mathcal{H}^{aff}_{q,R}(j)$.
%

By the previous paragraph, the biadjoint functors $E,F$ together with $X\in \operatorname{End}(F), T\in \operatorname{End}(F^2)$
define a type $A$ categorification on $\OCat^{\p}_{-e, R}$.

The same results are true in the undeformed setting, i.e,  for the category for $\OCat^{\p}_{-e}$. In fact,
they were proved in \cite{VV} in that setting but a generalization to the deformed setting is
straightforward.

We remark that on $\OCat_{\kappa,\bs}, \OCat^{\p}_{-e}$ the type A action becomes a categorical
$\hat{\sl}_e$-action with categorification functors $E_i,F_i, i\in \Z/e\Z$, where the functor $F_i$
is defined as the generalized eigenfunctor of $F$ with eigenvalue $q^i$.

Varagnolo and Vasserot proposed to consider the functor $\Hom(F^n \Delta(\varnothing), \bullet)$ from
$\OCat^{\p}_{-e}(n)$ to $\operatorname{End}(F^n\Delta(\varnothing))^{opp}\operatorname{-mod}$.
They observed that there is a natural  homomorphism from $\Hecke_q^{\bs}(n)$ to  $\operatorname{End}(F^n\Delta(\varnothing))^{opp}$
(we will recall why below). But, a priori, it is not clear why this homomorphism is an isomorphism. Also it
is unclear whether $F^n\Delta(\varnothing)$ is projective in the truncated category (definitely, it is
not projective in the whole parabolic category $\mathcal{O}$). We will check that
$\Hecke_{q,R}^{\bs}(n)\cong\operatorname{End}(F^n\Delta_R(\varnothing))^{opp}$ and that
$F^n\Delta(\varnothing)$ is projective in Section \ref{S_trunc_cat}. Proofs of both claims, see Section \ref{S_trunc_cat},  will be based on the categorical splitting for highest weight $\sl_2$-categorifications discovered
in \cite{str}.

\subsection{Hierarchy structure}\label{SS_hier}
In fact, the categorification functors $E,F$ on $\OCat_{\kappa,\bs,R},\OCat^{\p}_{-e,R}$
are nicely compatible with standardly filtered objects. Moreover, the specialized categories
are highest weight $\sl_2$-categorifications in the sense of \cite{str} for each pair of functors
$(E_i,F_i)$. For the category $\OCat_{\kappa,\bs}$ this was established in \cite[Section 4.3]{str}. We are going to
see that this is also true for $\OCat^{\p}_{-e}$. The first ingredient for the structure of a highest
weight categorical $\sl_2$-action is a certain structure on the poset of a highest weight category
called a hierarchy structure. So we start by establishing such a structure on $\Z^{\bs}$.

Fix a residue $i$ mod $e$. We start by equipping the poset $\Z^{\bs}$ of $\OCat^{\p}_{-e}$
with a hierarchy structure, see \cite[Section 3]{str}.

The first ingredient of a hierarchy structure as defined in \cite[Section 3.1]{str} is a family structure
that is a partition of a poset $\Lambda:=\Z^{\bs}$ into {\it families} $\Lambda_a$ for $a$ in some indexing
set $\mathfrak{A}$ together with  bijections $\sigma_a:\{+,-\}^{n_a}\rightarrow \Lambda_a$. Here the inverse
$\sigma_a^{-1}$ has to be increasing, where we equip  the set $\{+,-\}^{n_a}$ with a poset structure by setting
$(t_1,\ldots,t_{n_a})\preceq (t_1',\ldots,t_{n_a}')$ if the number of $-$'s among $t_1,\ldots, t_k$
does not exceed that for $t_1',\ldots,t_k'$ for any $k$, while the total number of $-$'s is the same.

By definition, two virtual multipartitions $\lambda,\mu$ lie in the same family if they can be obtained
from a single virtual multipartition $\nu$ by adding $i$-boxes. For $\lambda\in \Lambda_a$,  define
a tuple $\sigma_a^{-1}(\lambda)$ as follows. Number all addable and removable $i$-boxes $b$
in the increasing order with respect to the number $d(b)$ used in Section \ref{SS_part_poset}. For each such box $b$,
we write a $+$ if it is addable and a $-$ if it is removable. This collection of $+$'s and $-$'s is, by definition, $\sigma_{a}^{-1}(\lambda)$ (and so $n_a$ is the number of addable/removable $i$-boxes in $\lambda$).  It is clear that $\sigma_a^{-1}$ is increasing with respect to the order $\preceq$ on $\Lambda_a$ and $\preceq$ on $\{+,-\}^{n_a}$.

Let us define now the second component of a hierarchy structure -- a splitting structure.
Namely, to a family $\Lambda_a$ we need to assign a splitting $\Lambda=\Lambda^a_{>}\sqcup \underline{\Lambda}^a_+\sqcup
\underline{\Lambda}^a_-\sqcup \Lambda^a_{<}$ subject to certain axioms (S0)-(S4) from \cite[3.1]{str}. These axioms are
the following.
\begin{itemize}
\item[(S0)] For each $a$ the subsets $\Lambda^a_{<},\underline{\Lambda}^a_-\sqcup \Lambda^a_{<},
\underline{\Lambda}^a_+\sqcup \underline{\Lambda}^a_-\sqcup \Lambda^a_{<}\subset \Lambda$ are poset ideals.
\item[(S1)] $n_a=0$ if and only if $\underline{\Lambda}^a_-, \underline{\Lambda}^a_+=\varnothing$.
\item[(S2)] For each  $a,b$ the family $\Lambda_b$ is contained either in $\Lambda^a_{<}$ or
in $\Lambda^a_{>}$ or in $\Lambda^a_=:=\underline{\Lambda}^a_-\sqcup \underline{\Lambda}^a_+$.
Moreover, suppose $\Lambda_b\subset \Lambda_=^a$. An element $\lambda\in \Lambda_b$ is contained in
$\underline{\Lambda}^a_?$ if and only if the rightmost element of $\sigma^{-1}_b(\lambda)$ is $?$ (for $?=+,-$).
\item[(S3)] Let $a,b\in \mathfrak{A}$. If $\Lambda_b\subset \Lambda_=^a$, then $\underline{\Lambda}^b_?=
\underline{\Lambda}^a_?$ for $?=+,-$. The inclusion $\Lambda_b\subset \Lambda^a_>$ holds if and only if
$\Lambda_a\subset \Lambda^b_<$.
\item[(S4)] Let $a\in A$. Then there is a (automatically, unique)
poset isomorphism $\iota: \underline{\Lambda}^a_-\rightarrow \underline{\Lambda}^a_+$
that maps $\underline{\Lambda}^a_-\cap \Lambda_b$ to $\underline{\Lambda}^a_+\cap \Lambda_b$ such that if $\sigma_b^{-1}(\lambda)=t-$
for $t\in \{+,-\}^{n_b-1}$, then $\sigma_b^{-1}(\iota(\lambda))=t+$.
\end{itemize}


Let us now describe the four subsets $\Lambda^a_{>},\Lambda^a_{<},\underline{\Lambda}^a_+,\underline{\Lambda}^a_-$
in the case of interest. Let $b=(x,y,i)$ be the common smallest addable/removable box for the multipartitions in $\Lambda_a$.   For a multipartition $\lambda$ and a box $b'$, let $|\lambda|^{b'}$ denote the number of boxes $b''\in \lambda$ with $b''\sim b', d(b'')=d(b')$. For all boxes $b'$ with $b'\not\sim b$ or with $b'\sim b, b'<b$ the numbers $|\lambda|^{b'}$ do not depend on the choice of $\lambda\in \Lambda_a$. We remark that for  $\lambda\in \Lambda_a$ the number $|\lambda|^b$ takes one of the two consecutive values, say $s,s+1$.
Let $\underline{b}$ be the box just below $b$ (existing if $x\neq 1$).

Let $\Lambda^a_{>}$ consist of all multipartitions $\mu$ such that  one of the following conditions hold
\begin{itemize}
\item[(i)] There is a box $b'$ such that $d(b')>d(b)$ and we have $|\lambda|^{b''}=|\mu|^{b''}$ if $d(b'')>d(b')$
and $|\lambda|^{b'}>|\mu|^{b'}$.
\item[(ii)] We have $|\lambda|^{b'}=|\mu|^{b'}$ as long as $d(b')>d(b)$, while $|\mu|^b<s$.
\item[(iii)] We have $|\lambda|^{b'}=|\mu|^{b'}$ whenever $d(b')>d(b)$, as well as $|\mu|^b=s$. But there is
a box $\underline{b}'\leq \underline{b}$ such that $|\lambda|^{\underline{b}''}=|\mu|^{\underline{b}''}$
for any $\underline{b}''<\underline{b}'$ and $|\mu|^{\underline{b}'}<|\lambda|^{\underline{b}'}$.
\end{itemize}
Let $\underline{\Lambda}^a_+\sqcup \underline{\Lambda}^a_-$
consist of all multipartitions $\mu$ such that $|\mu|^b=s$ or $s+1$, and $|\mu|^{b'}=|\lambda|^{b'}$ for all $b'$ such that either
$d(b')<d(b)$ or $b'\leq \underline{b}$. Then automatically $b$ is an addable/removable box in any $\mu\in \underline{\Lambda}^a_+\sqcup \underline{\Lambda}^a_-$ and we form the subsets $\underline{\Lambda}^a_+, \underline{\Lambda}^a_-$ accordingly. Finally, let $\Lambda^a_<$
consist of the remaining partitions. The proof that (S0)-(S4) hold is completely parallel to that
given in \cite[Section 3.2]{str}. We remark however, that our definition of the subsets is different from \cite{str}.

The structure that we need to add to family and splitting structures to get a hierarchy structure is a family of subsets $\mathfrak{A}'\subset \mathfrak{A}$ together with posets $\Lambda(\mathfrak{A}')$  possessing both family and splitting
structures. The following axioms should hold.

\begin{itemize}
\item[(H0)] If $(\mathfrak{A}', \Lambda(\mathfrak{A}')), (\mathfrak{A}'', \Lambda(\mathfrak{A}''))\in \mathfrak{H}$
and $\mathfrak{A}'=\mathfrak{A''}$, then $\Lambda(\mathfrak{A}')=\Lambda(\mathfrak{A}'')$. Further, either one of
the subsets $\mathfrak{A}',\mathfrak{A}''$ is contained in the other, or $\mathfrak{A}',\mathfrak{A}''$ are disjoint.
\item[(H1)] $(\mathfrak{A},\Lambda)\in \mathfrak{H}$. If $(\mathfrak{A}', \Lambda(\mathfrak{A}'))\in \mathfrak{H}$, then, for any $a\in\mathfrak{A}'$, there is a splitting structure on $\underline{\Lambda(\mathfrak{A}')}^a$ such that
$\left((\underline{\mathfrak{A}}')^a, \underline{\Lambda(\mathfrak{A}')}^a\right)\in \mathfrak{H}$. Moreover,
every element $(\mathfrak{A}'',\Lambda(\mathfrak{A}''))$ is obtained from $(\mathfrak{A},\Lambda(\mathfrak{A}))$
by doing several steps as in  the previous sentence.
\item[(H2)] Any descending chain of embedded subsets in $\mathfrak{H}$ terminates.
\end{itemize}

\subsection{Highest weight categorical actions}\label{SS_categorif2}
Now let us show that $\OCat^{\p}_{-e}$ is a highest weight categorification in the sense of \cite{str}.

Choose $A\in \Z^{\bs}$ and represent $A$ as an $m$-tuple of integers (and not as a virtual
multipartition). According to \cite[Lemma A2.10]{VV}, the object $F\Delta_{R}(A)$
has a filtration whose successive subquotients are $\Delta_{R}(A^i)$, with $A^i:=A+\epsilon_i$, where $i$
is any index with $A_i\in \Z^{\bs}$. From the description of $X$ given in Section \ref{SS_categorif1} and ordering considerations,
it follows that $X$ preserves the filtration  and acts on the  subquotient $\Delta_{R}(A^i)$ by $\exp(2\pi \sqrt{-1} (a_i+x_i)(x_0-1/e))$.

On the level of virtual multipartitions, the standard subquotients of $F_i\Delta(\lambda)$ correspond to all
multipartitions obtained from $\lambda$ by adding an $i$-box.


Together with the considerations of  Section \ref{SS_hier}, this means that the functors
$F_i,E_i$ define a highest weight categorical $\sl_2$-action (with respect to the hierarchy
structure defined in Section \ref{SS_hier}) on the category $\OCat^{\p}_{-e}$ (modulo
the difference that the projectives lie not in the category but rather in its pro-completion,
but this difference actually does not matter). By definition, a highest weight categorification is a highest
weight category $\Cat$ with a highest weight poset $\Lambda$ equipped with a hierarchy structure,
and with an $\sl_2$-categorification, where the highest weight and the categorification  structures
are related via the following two axioms, see \cite[Section 4.2]{str}:
\begin{itemize}
\item[(i)] $F_i\Delta(\lambda)$ admits a filtration whose successive quotients are $\Delta(\lambda^1),\ldots,\Delta(\lambda^k)$,
where the elements $\lambda^1,\ldots,\lambda^k$ are determined as follows. Set $t=\sigma_a^{-1}(\lambda)$ and let
$j_1>j_2>\ldots>j_k$ be all indexes such that $t_{j_i}=+$. Then $\lambda^i:=\sigma_a(t^i)$, where $t^i$ is obtained
from $t$ by replacing the $j_i$-th element with a $-$.
\item[(ii)] $E_i\Delta(\lambda)$ admits a filtration whose successive quotients are $\Delta(\bar{\lambda}^1),\ldots,\Delta(\bar{\lambda}^l)$, where the elements $\bar{\lambda}^1,\ldots,
    \bar{\lambda}^l$ are determined as follows. Set $t=\sigma_a^{-1}(\lambda)$ and let
$j_1<j_2<\ldots<j_l$ be all indexes such that $t_{j_i}=-$. Then $\bar{\lambda}^i:=\sigma_a(\bar{t}^i)$, where $\bar{t}^i$ is obtained
from $t$ by replacing the $j_i$-th element with a $+$.
\end{itemize}
Recall that we work over $\C$, which is an uncountable field, so we do not need to impose any other assumptions,
see \cite[Section 4.1]{str}.

%

\subsection{Crystals}\label{SS_crystals}
In this subsection we will define certain $\hat{\sl}_e$-crystal structures: two, ``dual'' to each other, on $\mathcal{P}_\ell$ and
one on $\Z^{\underline{s}}$. Then we recall the results of \cite{cryst} on the crystals of highest weight $\sl_2$-categorifications.

Recall that an $\sl_2$-crystal on a set $\Lambda$ is a 5-tuple of maps $\tilde{e},\tilde{f}:\Lambda\rightarrow \Lambda \sqcup\{0\}$
and $h_+,h_-:\Lambda\rightarrow \Z_{\geqslant 0}, \wt:\Lambda\rightarrow \Z$ with the following properties:
\begin{itemize}
\item[(i)] $\wt=h_+-h_-$.
\item[(ii)] $\tilde{f}\lambda=0$ if and only if $h_+(\lambda)=0$. Similarly, $\tilde{e}\lambda=0$ if and only if $h_-(\lambda)=0$.
\item[(iii)] For $\lambda,\lambda'\in \Lambda$ the equality $\tilde{e}\lambda=\lambda'$ is equivalent to $\tilde{f}\lambda'=\lambda$.
\item[(iv)] If $\lambda'=\tilde{e}\lambda$, then $h_+(\lambda')=h_+(\lambda)+1, h_-(\lambda')=h_-(\lambda)-1$.
\end{itemize}
We remark that the functions $h_\pm,\wt$ are uniquely recovered from $\tilde{e},\tilde{f}$ and so below we do not
mention them as a part of a crystal structure.

By an $\hat{\sl}_e$-crystal on $\Lambda$ we will simply mean a collection $(\tilde{e}_i,\tilde{f}_i)_{i=0}^{e-1}$ of maps
$\Lambda\rightarrow \Lambda\sqcup\{0\}$ such that $\tilde{e}_i,\tilde{f}_i$ form an $\sl_2$-crystal. Since we are just going to deal with three particular crystal structures we do not care about compatibility relations between the $\sl_2$-crystals for different $i$. We will often write $h_{i,-},h_{i,+}$ for the functions $h_-,h_+$ defined for the residue $i$.

On the set $\mathcal{P}_\ell$ we will have two structures of $\hat{\sl}_e$-crystals. The first one, $(\tilde{e}_i,\tilde{f}_i)_{i=0}^{e-1}$
will be called the {\it usual} one, this is the structure provided by Uglov's dual canonical basis on the
representation of $U_q(\hat{\sl}_e)$ in the higher level Fock space corresponding to the multi-charge ${\underline{s}}$.
This basis was defined in \cite{Uglov}.
We will also consider a {\it dual} structure $(\tilde{e}_i^*,\tilde{f}_i^*)$
(corresponding to the canonical basis).
To define $\tilde{e}_i,\tilde{f}_i, \tilde{e}_i^*,\tilde{f}_i^*$ on $\lambda\in \mathcal{P}_i$
we will need to recall the notion of the  {\it $i$-signature} of $\lambda$ that will be an ordered collection of $+$'s and $-$'s.

For  a residue $i$ modulo $e$, we say that a box $b$  is an $i$-box if $\cont(b)-i$ is divisible by $e$.
A box $b$ lying in $\lambda$ is {\it removable}  if $\lambda\setminus b$ is still a multipartition. Similarly,
a box $b$ lying outside $\lambda$ is {\it addable} for $\lambda$ if $\lambda\sqcup b$ is still a multipartition.
To get the $i$-signature of $\lambda$
we list addable and removable $i$-boxes in $\lambda$ in the increasing order with respect to $d(b)=-\frac{\ell}{e}\cont(b)-j$,
where $b\in \lambda^{(j)}$. In other words, a box $b$ occurs to the left of $b'$ in our list if either $\cont(b)>\cont(b')$
or $\cont(b)=\cont(b')$ but $j>j'$, where $b\in \lambda^{(j)}, b'\in \lambda^{(j')}$. The $i$-signature $t=(t_1,\ldots,t_k)$ is obtained
from this list by replacing an addable box with a $+$ and a removable box with a $-$.

We are going to introduce two reduction procedures for $i$-signatures. The first one is as follows. Consider the sets $I_\pm$ of indexes
$j$ such that $t_j=\pm$. On each step we are going to remove a consecutive pair $-+$, i.e.,  to remove $j\in I_+, j'\in I_-$ if $j'<j$
and no elements between $j$ and $j'$ remain. We stop when no removal is possible. The terminal pair $(I_+,I_-)$
(that is easily seen to be independent of the order of removals) is called the {\it reduced $i$-signature} of $t$
(or of $\lambda$). It clearly has a property that all elements of $I_+$ are less than all elements of $I_-$.
We will write $(I_+(t),I_-(t))$ for the reduced signature of $t$.


To define the dual crystal structure we use a similar procedure. The difference here
is that we remove consecutive pairs $+-$, i.e., we remove
$j\in I_+, j'\in I_-$ if $j'>j$ and there are no elements between $j$ and $j'$ left.
In this way we will get the {\it dual reduced signature} $(I_+^*(t), I_-^*(t))$, it has a
property that the elements of $I^*_+(t)$ are bigger than the elements of $I^*_-(t)$.
For example, for $t=(+++-++---)$ we have $I_+(t)=\{1,2,3,6\}, I_-(t)=\{7,8,9\}, I^*_+(t)=\{1\}, I^*_-(t)=\varnothing$.

The maps $\tilde{e}_i, \tilde{f}_i,\tilde{e}_i^*, \tilde{f}_i^*$ are defined as follows.
The map $\tilde{f}_i$ adds an $i$-box in the position corresponding to the largest element
of $I_+(t)$ if the latter is non-empty, or sends $\lambda$ to $0$ else. Similarly,
$\tilde{e}_i$ removes the box corresponding to the minimal  element of $I_-(t)$ or sends
$\lambda$ to $0$ if the latter set is empty. The corresponding maps $h_+,h_-$ are defined
by $h_+(\lambda)=|I_+(t)|, h_-(\lambda)=|I_-(t)|$. Next, the map $\tilde{f}^*_i$
adds an $i$-box in the position corresponding to the minimal element of $I^*_+(t)$.
The map $\tilde{e}^*_i$ removes the box in the position corresponding to the maximal element of
$I^*_-(t)$.

In fact, the dual and the usual crystal structures are related via a duality. Namely,
consider the multicharge ${\underline{s}}^\dagger=(-s_{\ell},-s_{\ell-1},\ldots,-s_1)$. Let $\tilde{e}^\dagger_i, \tilde{f}^\dagger_i$
be the crystal operators on $\mathcal{P}_\ell$ associated to this multicharge. For a box $b=(x,y,i)$
consider a box $b^\dagger=(y,x,\ell+1-i)$. Also for a multipartition $\lambda$ consider a multipartition
$\lambda^\dagger:=(\lambda^{(\ell)t},\lambda^{(\ell-1)t},\ldots, \lambda^{(1)t})$, where the superscript $t$
means the usual transposition of Young diagrams. We have $b\prec_{\underline{s}}b'$ if and only if $b'^\dagger\prec_{{\underline{s}}^\dagger} b^\dagger$.
So $\lambda\mapsto \lambda^\dagger$ is an order reversing bijection. Also we remark that the $i$-signature
of $\lambda$ is obtained from the $-i$-signature of $\lambda^\dagger$ by reversing the order of elements.
It follows $\tilde{e}_{-i} \lambda^\dagger= (\tilde{e}^*_{i} \lambda)^\dagger$ and
$\tilde{f}_{-i} \lambda^\dagger= (\tilde{f}^*_{i} \lambda)^\dagger$.

We say that $\lambda$ is {\it singular} if $\tilde{e}_i \lambda=0$ for all $i$
and {\it cosingular} if $\tilde{e}^*_i\lambda=0$ for all $i$.

We would like to remark that the weight functions $\wt$ are the same for the usual and the dual crystal
structure. We say that two multipartitions $\lambda,\mu$  lie in the same block if $|\lambda|=|\mu|$, and
the $e$-tuples of their weight functions coincide (the $e$-tuple is referred to as the weight of a multipartitions).
Equivalently, $\lambda$ and $\mu$ lie in the same block if the number of $i$-boxes in $\lambda$
equals the number of $i$-boxes in $\mu$ for any residue $i$.

Finally, let us define an $\hat{\sl}_e$-crystal structure on $\Z^{\bs}$. This is done absolutely analogously
to the usual crystal structure above but we deal with virtual multi-partitions rather than with ordinary
ones.

Now let us proceed to a categorical interpretation of the usual crystal structure.

In \cite{cryst}, we have determined crystal structures for highest weight $\sl_2$-categorifications.
Let us recall the general construction of the crystal of an $\sl_2$-categorification, essentially due to Chuang
and Rouquier, \cite[Proposition 5.20]{CR}. The crystal structure is on the set of isomorphisms classes of simple objects.
For a simple $L$, the object $EL$ has simple head (the maximal semisimple quotient) and socle
(the maximal semisimple subobject) and they are isomorphic. The same is true for $FL$. By definition, $\tilde{e}L$ (resp., $\tilde{f}L$)
is  the socle of $EL$, resp., $FL$, if that object is nonzero or zero else. Of course, for a highest weight $\sl_2$-categorification
the set of iso-classes of simples is naturally identified with the highest weight
poset $\Lambda$.

Also we have a crystal structure on $\{+,-\}^n$ defined using the ``usual'' cancelation rule
for $+$'s and $-$'s, see above. The main result of \cite{cryst}
is that, for a highest weight $\sl_2$-categorification $\Cat$ with poset $\Lambda$,
each family $\Lambda_a$  is a subcrystal (which is easy), and each map $\sigma_a$
is an isomorphism of crystals (which is harder).

A corollary of the previous paragraph is that the crystal of the $\hat{\sl}_e$-categorification on $\OCat_{\kappa,\bs}$
is the same as the usual crystal described above in this subsection. The crystal
structure on $\OCat^{\p}_{-e}$ is obtained in a similar way, if we replace the usual multipartitions
with virtual ones.

We will need an alternative characterization of the simples lying in certain crystal components.

\begin{Lem}\label{Lem_cryst}
Let $\Cat=\bigoplus_{j=0}^\infty \Cat(j)$ be a  highest weight $\hat{\sl}_e$-categorification
of the Fock space $\mathcal{F}^{\bs}$. Let $\Par^{(r)}_\ell$ denote the union of crystal components
containing a multi-partition of $\leqslant r$. Pick $\lambda\in \Par_\ell$. Then $\lambda\in \Par_\ell^{(r)}$
if and only if $E^{|\lambda|-r}L(\lambda)\neq 0$.
\end{Lem}
\begin{proof}
If $\lambda\in \mathcal{P}_\ell^{(r)}$, then $E^{|\lambda|-r} L(\lambda)\neq 0$. To see that
$E^{|\lambda|-r} L(\lambda)=0$ for $\lambda\not\in \mathcal{P}_\ell^{(r)}$ we notice that  $\mathcal{P}_\ell$
is  the crystal  of the Fock space.
Let $\mathcal{F}_{\bs}'$ be the graded $\hat{\sl}_e$-stable complement
to the sum of the  irreducibles in $\mathcal{F}_{\bs}$ spanned by vectors of degree $\leqslant r$.
Of course, $\mathcal{P}_{\ell}\setminus \mathcal{P}_\ell^{(r)}$ is the crystal of  $\mathcal{F}_{\bs}'$.
Then all elements in $\mathcal{F}'_{\bs}(j)$
are annihilated by $(e_0+\ldots+e_{e-1})^{j-r}$. For any $\lambda\in \mathcal{P}_\ell(j)\setminus \mathcal{P}_\ell^{(r)}$
we have $[L(\lambda)]\in \mathcal{F}'_{\bs}(j)$. Indeed, $\mathcal{F}_{\bs}'$ is the $K_0$-space for the kernel
of the quotient morphism defined by the sum of the projectives of the form $F^k P(\lambda)$, where
$\lambda$ is singular in $\Par_\ell(\leqslant r)$. Hence our claim.
\end{proof}

\subsection{More structural results about highest weight $\sl_2$-categorifications}\label{SS_sl2_str}
Let $\Cat$ be a highest weight $\sl_2$-categorification with respect to a hierarchy structure on a
poset $\Lambda$ as defined in \cite{str}. Let $E,F$ denote the categorification functors. In the example
above, $E:=F_i, F:=E_i$.

The splitting structure that is a part of a hierarchy structure has a categorical counterpart defined
and studied in \cite[Section 5]{str}. Namely, pick an index $a$ and consider the splitting
$\Lambda=\Lambda^a_{<}\sqcup \underline{\Lambda}^a_{-}\sqcup \underline{\Lambda}^a_+\sqcup \Lambda^a_{>}$.

A subquotient in a highest weight category associated to a poset interval  inherits a highest weight
  structure. More precisely, pick a poset ideal $\Lambda'\subset\Lambda$.
Consider the Serre subcategory $\Cat(\Lambda')\subset \Cat$ spanned by $L(\lambda), \lambda\in \Lambda'$.
This is a highest weight subcategory with respect to the standards $\Delta(\lambda),\lambda\in \Lambda'$.
For two such subsets $\Lambda'\subset \Lambda''$ we can form the quotient category $\Cat(\Lambda'')/\Cat(\Lambda')$
that is also a highest weight category, the standards are the images of $\Delta(\lambda), \lambda\in
\Lambda''\setminus \Lambda'$.

Set $\Lambda':=\Lambda^a_{<}, \Lambda'':=\Lambda^a_{<}\sqcup \underline{\Lambda}^a_-$. In \cite[Section 5.3]{str} we have equipped
the highest weight category $\underline{\Cat}^a_-:=\Cat(\Lambda'')/\Cat(\Lambda')$ with a highest weight $\sl_2$-categorification
structure (with respect to the hierarchy structure on $\underline{\Lambda}^a_-$). Namely, the functor $E$
preserves both $\Cat(\Lambda'),\Cat(\Lambda'')$ and so induces an exact functor $E_-$ on $\underline{\Cat}^a_-$.
It turns out that there is a biadjoint functor $F_-$ to $E_-$ that provides a categorification structure.

Using a categorical splitting construction in \cite[Section 5.4]{str} we have established a certain increasing filtration
on $\Cat$ with filtration subcategories of the form $\Cat(\Lambda')$. Subsequent quotients of this filtration
are so called basic categorifications. These are simplest possible highest weight $\sl_2$-categorifications.
A highest weight poset of a basic categorification is $\{+,-\}^n$ and the order is defined as in
Section \ref{SS_hier}. There is only one family and the map $\sigma$ is the identity. Also there is only one possible
splitting structure and only one possible hierarchy structure on $\{+,-\}^n$.

Finally, we will need some information on the structure of $ET(\lambda)$ and $FT(\lambda)$ obtained in
\cite[Proposition 7.1]{str}.  The following is a direct corollary of \cite[Proposition 7.1]{str}.

\begin{Lem}\label{Lem:Tilt_str}
$T(\tilde{e}^*\lambda)$ is a direct summand of $E T(\lambda)$ and $T(\tilde{f}^*\lambda)$
is a direct sum of $FT(\lambda)$.
\end{Lem}

\section{Truncated parabolic categorification}\label{S_trunc_cat}
\subsection{Categorification functors on $\OCat^{\p}_{-e}(\leqslant n)$}\label{SS_trunc_cat}
Recall the truncated category $\OCat^{\p}_{-e}(\leqslant n)\subset \OCat^{\p}_{-e}$.
Also recall that the functor $E_i$ sends $\Delta(\lambda), \lambda\in \Z^{\bs},$ to
a filtered object, whose successive filtration quotients are standards labeled
by virtual multipartitions obtained from $\lambda$ by removing an $i$-box.

Below we always assume that $n<s_i$ for all $i$. Let $\lambda$ be a genuine multipartition.
Then removing an $i$-box with $i\neq 0$ from $\lambda$ we still get a genuine multipartition.
So $E_i$ preserves $\OCat^{\p}_{-e}(\leqslant n)$ if $i\neq 0$. However, $E_0$ does not. Indeed,
the set of removable boxes of the virtual multipartition $\tilde{\lambda}$ corresponding to $\lambda$
consists of the removable boxes of $\lambda$ together with the $\ell$ boxes $(0,s_i,i), i=1,\ldots,\ell$. All these $\ell$ boxes
are  0-boxes. These are the smallest $\ell$ removable boxes in the virtual multipartition. The functor $E_0$ can remove one of
these $\ell$ boxes and so $\ell$ standard subquotients in a filtration of $E_0\Delta(\lambda)$ do not belong
to $\OCat^{\p}_{-e}(\leqslant n)$.

Let $\Lambda_1$ be the set of all actual partitions $\lambda=(\lambda^{(1)},\ldots,\lambda^{(\ell)})$ such that
$\lambda^{(i)}$ has strictly less than $s_i$ rows. The boxes $(0,s_i,i), i=1,\ldots,\ell$ are still the smallest
removable 0-boxes of the corresponding multipartition $\tilde{\lambda}$.
We claim that $\Lambda_1$ is one of the posets of the form $\Lambda(\mathfrak{A}')$
appearing in the hierarchy (with the $\ell$ boxes frozen).  Namely, it is obtained by doing the $\ell$
iterations of  the  transformation $\Lambda\mapsto \underline{\Lambda}^a_-$, where each time $a$ is the
index for a family containing $\tilde{\lambda}$ (this will freeze precisely the $\ell$ boxes $(0,s_i,i)$).

Let $\Cat_1$ be the highest weight subquotient category of $\Cat$ corresponding to $\Lambda_1$ so that $\OCat^{\p}_e(\leqslant n)$
embeds into $\Cat_1$ as a highest weight subcategory.
Now the categorical splitting construction recalled in Section \ref{SS_sl2_str} (applied $\ell$ times)
equips $\Cat_1$ with a structure of a highest weight $\mathfrak{sl}_2$-categorification with respect to the induced hierarchy structure
on $\Lambda_1$. Let us remark that $\OCat^{\p}_{-e}(\leqslant n)$ embeds into $\Cat_1$
as a highest weight subcategory. In particular, we get a truncated functor $E_0$ on $\OCat^{\p}_{-e}(\leqslant n)$
that now preserves the subcategory.

So now we have functors $E_i:\OCat^{\p}_{-e}(j)\rightarrow \OCat^{\p}_{-e}(j-1), i=0,\ldots,e-1$ defined for $j\in \{0,1,\ldots,n\}$
(we set $\OCat^{\p}_{-e}(-1)=0$) and also functors $F_i:\OCat^{\p}_{-e}(j)\rightarrow
\OCat^{\p}_{-e}(j+1)$ defined for $j\in \{0,1,\ldots,n-1\}$. All biadjointness properties between $E_i,F_i$ still hold.
For $i\neq 0$ these properties holds  because of the embedding $\OCat^{\p}_{-e}(\leqslant n)\hookrightarrow \OCat^{\p}_{-e}$
and for $i=0$ -- because of the embedding $\OCat^{\p}_{-e}(\leqslant n)\hookrightarrow \Cat_1$. Also the embedding
$\OCat^{\p}_{-e}(\leqslant n)$ produces natural transformations $X$ of $F:=\bigoplus_{i=0}^{e-1}F_i: \OCat^{\p}_{-e}(j)
\rightarrow \OCat^{\p}_{-e}(j+1)$ (for $j<n$) and $T$ of $F^2:\OCat^{\p}_{-e}(j)\rightarrow \OCat^{\p}_{-e}(j+2)$
(for $j<n-1$) that satisfy the Hecke relations.

So the structure on $\OCat^{\p}_{-e}(\leqslant n)$ that we get is almost that of an $\hat{\sl}_e$-categorification
with the difference that a genuine categorification would categorify the whole Fock space, while
$\OCat^{\p}_{-e}(\leqslant n)$ categorifies just the sum of graded components with degrees from $0$ to $n$.
We will call such a structure a {\it restricted categorification}.

We remark that all results mentioned in Section \ref{SS_sl2_str} hold for $\OCat^{\p}_{-e}(\leqslant n)$ in the following
sense. First, one can still define the (restricted) crystal for $\OCat^{\p}_{-e}(\leqslant n)$ as before.
The crystal coincides with the restriction of the crystal for $\mathcal{P}_\ell$. The reason is that for $i\neq 0$
the operators $\tilde{e}_i,\tilde{f}_i$ come from $\OCat^{\p}_{-e}$, while for $i=0$ they come
from $\Cat_1$. A straightforward analog of Lemma \ref{Lem_cryst} holds.
Thanks to the embeddings
$\OCat^{\p}_e(\leqslant n)\hookrightarrow \OCat^{\p}_{-e},\Cat_1$, Lemma \ref{Lem:Tilt_str}
still holds.

We say that a highest weight category $\mathcal{O}$ equipped with a
categorical action of $\hat{\mathfrak{sl}}_e$ is a highest weight categorification
of the level $\ell$ Fock space $\mathcal{F}^{\bs}$ if the following holds:
\begin{itemize}
\item  The highest weight poset of $\mathcal{O}$ is the poset of partitions $\Par_\ell$ with the order
defined by $\bs$ as in Section \ref{SS_part_poset}.
\item The categorical $\hat{\mathfrak{sl}}_e$-action
categorifies the action on the Fock space in such a way that the class $[\Delta(\lambda)]$
is the basis vector of $\mathcal{F}^{\bs}$ corresponding to $\lambda$.
\item For each $i\in \Z/e\Z$, the functors $E_i,F_i$ give a highest weight categorical
action on $\mathcal{O}$, see Section \ref{SS_categorif2}.
\end{itemize}
So both $\OCat_{\kappa,\bs}$ and $\OCat^{\p}_{-e}(\leqslant n)$ are highest weight
categorifications of $\mathcal{F}^{\bs}$ (the second category is a restricted categorification).

Similarly, we can define deformed (over $R$) highest weight categorifications
of $\mathcal{F}^{\bs}$.


\subsection{Object $F^n \Delta(\varnothing)$}\label{SS_proj}
In this section we are going to investigate the structure of the object $F^j\Delta(\varnothing)\in
\OCat^{\p}_{-e}(j)$. First of all,
let us observe that the transformations $X,T$ of $F,F^2$ that are a part of the categorification structure
on $\OCat^{\p}_{-e}$ give rise to a homomorphism $\Hecke^{aff}_{q,R}(j)\rightarrow \End(F^j\Delta_R(\varnothing))^{opp}$.

The main result is
the following proposition.

\begin{Prop}\label{Prop:FnDelta}
Assume as before that $n<s_i$ for all $i=1,\ldots,\ell$. Then the following claims hold:
\begin{enumerate}
\item $F^j\Delta_R(\varnothing)$ is a projective object in $\OCat^{\p}_{-e,R}(j)$.
\item The homomorphism $\Hecke^{aff}_{q,R}(j)\rightarrow \End(F^j\Delta_R(\varnothing))^{opp}$
factors through an isomorphism $\Hecke_{q,R}^{\bs}(j)\xrightarrow{\sim}\End(F^j\Delta_R(\varnothing))^{opp}$.
\end{enumerate}
\end{Prop}
\begin{proof}
To show that
the homomorphism $\Hecke^{aff}_{q,R}(j)\rightarrow \End(F^j\Delta_R(\varnothing))^{opp}$ factors
through $\Hecke_{q,R}^{\bs}(j)$ it is enough to consider the case of $j=1$. In this case
the claim follows from the form of the successive quotients of the filtration on $F_i\Delta_R(\varnothing)$,
see Section \ref{SS_categorif2}.

Recall that  $\OCat^{\p}_{-e,R}(\leqslant n)$ is a flat formal  deformation
of $\OCat^{\p}_{-e}(\leqslant n)$. The object $F^j\Delta_R(\varnothing)$ is flat. Therefore
it is enough to prove all the other claims after base change to the closed point.

The claim that $F^j\Delta(\varnothing)$ is projective follows from the observation that $\Delta(\varnothing)$ is
and that $F$ admits a biadjoint functor, the functor $E_0\oplus E_1\oplus\ldots\oplus E_{e-1}$
from the previous subsection.

So it remains to show that $\Hecke_q^{\bs}(j)\rightarrow \End(F^j\Delta(\varnothing))^{opp}$
is an isomorphism. Let $\mathcal{L}^{\bs}$ be Rouquier's irreducible $\hat{\sl}_e$-categorification with
highest weight $\sum_{i=1}^\ell \Lambda_{s_i}$, see \cite[5.1.2]{Rouquier_2Kac}.
It has a natural grading $\mathcal{L}^{\bs}=\bigoplus_{i=0}^\infty \mathcal{L}^{\bs}(i)$
according to the degree. Let ${\bf 1}_{\bs}$ be the indecomposable object in $\mathcal{L}^{\bs}(0)$.
Despite the fact that our categorification $\OCat^{\p}_{-e}(\leqslant n)$ is restricted, Rouquier's
construction in \cite[5.1.2]{Rouquier_2Kac} still works and produces a unique morphism of restricted
categorifications  $\bigoplus_{i=0}^n \mathcal{L}^{\bs}(i)\rightarrow \OCat^{\p}_{-e}(\leqslant n)$ that maps
${\bf 1}_{\bs}$ to $\Delta(\varnothing)$. The proof of \cite[Lemma 5.4]{Rouquier_2Kac}
carries to our situation verbatim and implies that the functor
$\bigoplus_{i=0}^n \mathcal{L}^{\bs}(i)\rightarrow \OCat^{\p}_{-e}(\leqslant n)$
is fully faithful. So the homomorphism $\Hecke_q^{\bs}(j)\rightarrow \End(F^j\Delta(\varnothing))$
is the composition of $\Hecke_q^{\bs}(j)\rightarrow \End(F^j{\bf 1}_{\bs})^{opp}$ and the isomorphism
$\End(F^j{\bf 1}_{\bs})\xrightarrow{\sim} \End(F^j\Delta(\varnothing))$. But, thanks to \cite[Sections 5.3.7,  5.3.8]{Rouquier_2Kac},
the homomorphism $\Hecke_q^{\bs}(j)\rightarrow \End(F^j{\bf 1}_{\bs})^{opp}$ is an isomorphism.
\end{proof}

Let us specify the indecomposable summands of $F^j\Delta(\varnothing)$ in $\OCat^{\p}_{-e}(\leqslant n)$.
The same description is true for $\OCat_{\kappa,\bs}$.

\begin{Prop}\label{Prop:cryst}
For $\lambda\in \mathcal{P}_\ell(j)$, an indecomposable projective $P(\lambda)$ appears in
$F^j \Delta(\varnothing)$ if and only if $\lambda$ lies in the connected  component $\mathcal{P}_\ell^0$ of $\varnothing$
in the crystal of $\mathcal{P}_\ell$.
\end{Prop}
\begin{proof}
$P(\lambda)$ appears in $F^j\Delta(\varnothing)$ if and only if $\Hom(F^j\Delta_0,L(\lambda))\neq 0$ if and only
if $E^j L(\lambda)\neq 0$. Now our claim is a special case of Lemma \ref{Lem_cryst}.
\end{proof}

\section{Faithfulness via combinatorics}\label{S_faith_comb}
Fix $n>0$. We are working with a restricted highest weight $\hat{\sl}_e$-categorification
$\Cat(\leqslant N):=\bigoplus_{j=0}^{N}\Cat(j)$ of $\mathcal{F}^{\bs}$, where $N\gg n$.
Our goal in this section is to show that a certain combinatorial condition guarantee
(-1)-faithfulness of appropriate quotient functors. We then check
that the condition holds in a special case.

\subsection{Combinatorial condition}\label{SS_comb_conj}
Fix an element $w$ in the affine symmetric group $\hat{\mathfrak{S}}_e$ together with some reduced
decomposition $w=\sigma_{i_k}\sigma_{i_{k-1}}\ldots \sigma_{i_1}$.
For example, we will often consider  cycles $w=\sigma_{j'}\sigma_{j'+1}\ldots \sigma_{j-1}\sigma_j$.

For $\lambda$ with $h_{i,-}(\lambda)=0$ we define $\sigma_i\lambda:=\tilde{f}_i^{h_{i,+}(\lambda)}\lambda$.
If $h_{i,-}(\lambda)\neq 0$,  we say that $\sigma_i\lambda$ is undefined. So we can define an element $w\lambda$
(a priori, depending on the reduced decomposition of $w$) or say that this element is undefined.
Recall that any two reduced decompositions are obtained from one another by applying braid
moves, i.e., replacing a fragment $\sigma_i \sigma_{i+1}\sigma_i$ with $\sigma_{i+1}\sigma_i \sigma_{i+1}, i \in \Z/p\Z,$ and vice versa.
From here it is not difficult to see that $w\lambda$ depends only on $w$ and not on the reduced
expression. If $\lambda$ is singular, then, as we will check rigorously below,
$w\lambda$ is always defined.

Similarly, if $\mu$ is cosingular, we can define the element $w^*\mu$ by using the dual crystal
operators $\tilde{f}^*$. We remark that if $\lambda$ and $\mu$ are in the same block, then so
are $w\lambda, w^*\mu$.

Let us state a  condition on pairs $(\lambda,\mu)$, where $\lambda$
is singular and $\mu$ is consingular.

\begin{itemize}
\item[$(\mathfrak{C}_{\lambda,\mu})$] There is $w\in \hat{\mathfrak{S}}_e$ such that $w\lambda\not\preceq w^*\mu$.
\end{itemize}

\subsection{(-1)-faithfulness from  condition $\mathfrak{C}$}\label{SS_comb_to_-1faith}
Here we will relate faithfulness of quotient functors to the condition $\mathfrak{C}$.
The main result is Proposition \ref{Prop:-1_faith}.

First of all, let us check that, for a singular $\lambda$, the element $w\lambda$ is well-defined.

\begin{Lem}\label{Lem:refl_well_def}
Let $\lambda\in \mathcal{P}_{\ell}$ be singular. Let $w\in \hat{\mathfrak{S}}_e$ and let
$w=\sigma_{i_k}\ldots \sigma_{i_2}\sigma_{i_1}$ be some reduced expression. Then
$\tilde{e}_{i_{j+1}}(\sigma_{i_{j}}\ldots \sigma_{i_1}\lambda)=0$.
\end{Lem}
\begin{proof}
Consider the class $[L(\lambda)]$ of $L(\lambda)\in \OCat_{\kappa,\bs}$ in the Fock space $[\OCat_{\kappa,\bs}]$.
Define an element $w[L(\lambda)]$ by applying some lifting of $w$ in the Kac-Moody group to $[L(\lambda)]$.
This element is well-defined up to a sign.
The operator $[E_{i_j}]$ does annihilate $\sigma_{i_{j-1}}\ldots \sigma_{i_1}[L(\lambda)]$ because the element
$[L(\lambda)]$ is singular.  So to prove the lemma, it suffices to check $[L(\sigma_{i_j}\ldots \sigma_{i_1}\lambda)]$ is a multiple of
$\sigma_{i_j}\ldots \sigma_{i_1}[L(\lambda)]$.

This claim is proved by induction on $j$. The base is trivial and the step follows from the observation that $\sigma_{i_j}$ acts on both $[L(\sigma_{i_{j-1}}\ldots \sigma_{i_1}\lambda)]$ and $\sigma_{i_{j-1}}\ldots \sigma_{i_1}[L(\lambda)]$ by nonzero multiples of the same
power of $[F_{i_j}]$.
\end{proof}


\begin{Prop}\label{Prop:reduction}
Suppose $\lambda,\mu\in \mathcal{P}_{\ell}$ lie in the same block and $\tilde{e}_i \lambda=\tilde{e}^*_i\mu=0$.  Suppose
further that $\mathcal{P}_{\ell}(\leqslant N)$
contains the whole $i$-families of $\lambda,\mu$. Then we have $\dim\Ext^j(L(\lambda), T(\mu))=\dim\Ext^j(L(\sigma_i\lambda), T(\sigma_i^*\mu))$
for all $j$.
\end{Prop}
\begin{proof}
Set $k:=h_{i,+}(\lambda)$. So $\sigma_i\lambda:=\tilde{f}_i^k\lambda$  and therefore $L(\sigma_i\lambda)=F_i^{(k)}L(\lambda)$.
So $\Ext^j(L(\sigma_i\lambda), T(\sigma_i^*\mu))=\Ext^j(F_i^{(k)} L(\lambda), T(\sigma_i^*\mu))=\Ext^j(L(\lambda), E_i^{(k)}T(\sigma_i^*\mu))$. But $\mu= (\tilde{e}_i^{*})^k \sigma_i^*\mu$. Lemma \ref{Lem:Tilt_str} (applied $k$ times) implies that
$T(\mu)=T((\tilde{e}_i^{*})^k \sigma_i^*\mu)$ is a direct summand in $E_i^k T(\sigma_i^*\mu)$ and hence in $E^{(k)}_i T(\sigma_i^*\mu)$.
Therefore  $\Ext^j(L(\lambda), T(\mu))\hookrightarrow \Ext^j(L(\sigma_i\lambda), T(\sigma_i^*\mu))$. The inclusion in
the opposite direction is proved analogously.
\end{proof}

Now let us relate the condition $\mathfrak{C}$ to (-1)-faithfulness.

%
%



\begin{Prop}\label{Prop:-1_faith}
Suppose that $(\mathfrak{C}_{\lambda,\mu})$ holds for any singular $\lambda\in \mathcal{P}(\leqslant n)$ and any cosingular $\mu$. Let $\Cat(\leqslant N)$ be a restricted categorification of $\mathcal{F}_e^{\bs}$,
where $N$ is such that, for all singular $\lambda\in \mathcal{P}_\ell(\leqslant n)$ and cosingular $\mu\in
\mathcal{P}_\ell(|\lambda|)$, one can find $w$ in
$(\mathfrak{C}_{\lambda,\mu})$ such that $|w\lambda|\leqslant N$.
Consider the quotient functor $\pi$ corresponding to  $\Par_\ell^0(\leqslant n)$,
where $\Par_\ell^0$ is the connected component of $\varnothing$ in $\Par_\ell$.
Then this functor
is $(-1)$-faithful.
\end{Prop}
We will not need the bound on $N$ below. It is just needed for the proof.
\begin{proof}
Assume the converse: there is $\lambda\in \Par(\leqslant n)\setminus \Par^0$
such that $\Hom(L(\lambda),T(\mu))\neq 0$.
We  may assume that we have chosen $|\lambda|$ to be minimal with this property.
Note that this implies that $\lambda$ is singular. If not, then $L(e_i\lambda)$ lies in
the socle of $E_iT(\mu)$.

%
Let us show that if $\Ext^i(L(\lambda),T(\mu))\neq 0$ with singular $\lambda$, then $\mu$
is cosingular. Assume that there is $i$ such that $\mu':=\tilde{e}_i^*\mu\neq 0$. Then, by Lemma \ref{Lem:Tilt_str},
$T(\mu)$ is a direct summand of $F_i T(\mu')$ and hence $\Hom(L(\lambda),T(\mu))\hookrightarrow \Hom(L(\lambda), F_i T(\mu'))=\Hom(E_i L(\lambda), T(\mu'))=0$. Our claim is proved. In particular, if $L(\lambda)$ lies in the socle of $T(\mu)$, then $\mu$ is cosingular.
%
%

Pick $w$ as in $(\mathfrak{C}_{\lambda\mu})$. Let $w=\sigma_{i_1}\ldots \sigma_{i_k}$ be a reduced expression for $w$.
Applying   Proposition \ref{Prop:reduction} several times we see that
$\dim \Hom(L(\lambda),T(\mu))=\dim \Hom(L(w\lambda),T(w^*\mu))$.
The right hand side is zero because $L(w\lambda)$ is not a composition factor of $T(w^*\mu)$ as $w\lambda\not\prec w^*\mu$.
\end{proof}


\subsection{Checking the combinatorial condition: level 1}\label{SS:comb_1}
Here we are going to prove the following claim.
\begin{Prop}\label{Prop:comb_cond_holds}
$(\mathfrak{C}_{\lambda,\mu})$ holds
for all singular $\lambda$ and cosingular $\mu$ provided $\ell=1$.
\end{Prop}

Singular diagrams are $\lambda=(\lambda_1,\ldots,\lambda_d)$ such that each $\lambda_j$ is divisible by $e$.
Cosingular diagrams are transpose to singular ones. In other words, $\mu$ is cosingular if the multiplicity of
each part in $\mu$ is divisible by $e$.

We will work with the cycle $w=C_{j,n}=\sigma_{j-n+1}\ldots \sigma_{j-1}\sigma_j$. We will explicitly compute $w\lambda, w^*\mu$.

A crucial observation for our computation (which is no longer true for $\ell>1$) is that all crystal components
of $\mathcal{P}$ are isomorphic to the component of $\varnothing$ via a very easy isomorphism. Namely, for a singular
partition $\lambda=(\lambda_1,\ldots,\lambda_d)$ let $\mathcal{P}^\lambda$ denote its connected component in the crystal.
The following is well-known.

\begin{Lem}\label{Lem:cryst_iso}
The map $\mathcal{P}^{\varnothing}\rightarrow \mathcal{P}^\lambda$ that sends $\mu=(\mu_1,\ldots,\mu_d)$ to
$(\mu_1+\lambda_1,\ldots,\mu_d+\lambda_d)$ is an isomorphism of crystals.
\end{Lem}

As a corollary of this lemma, we get the formulas for $C_{j,n}\lambda$ and $C^*_{j,n}\mu$. We assume that
$s_1=0$ and so the residue of the box $(1,1)$ is $0$. Also we assume that $j=0$.
We introduce partitions $\xi_n$: by definition, $\xi_n=(n,n-(e-1),n-2(e-1),\ldots)^t$.
For two partitions $\lambda^1,\lambda^2$ we write $\lambda^1+\lambda^2$ for their componentwise
sum, $\lambda^1+\lambda^2=(\lambda^1_1+\lambda^2_1,\lambda^1_2+\lambda^2_2,\ldots)$.

\begin{Cor}\label{Cor:sing_cosing}
We have $C_{0,n}\lambda=\lambda+\xi_n$ and $C^*_{0,n}\mu=(\mu^t+\xi_n^t)^t$.
\end{Cor}
\begin{proof}
The proof for the usual crystal structure boils down to the case $\lambda=\varnothing$, thanks
to Lemma \ref{Lem:cryst_iso}. The proof that $C_{0,n}\varnothing=\xi_n$ is by induction on $n$,
for the induction step we need to notice that each time we apply $\sigma_j$, there are no removable $j$-boxes.
As for the dual crystal structure, we notice that $C_{0,n}^*\mu=(C_{0,n}^{\vee} \mu^t)^t$, where
$C_{0,n}^\vee=\sigma_{n-1}\ldots \sigma_{1}\sigma_0$. Then $C_{0,n}^\vee \varnothing=\xi_n^t$ for the same
reason as above. The proof of $C^*_{0,n}\mu=(\mu^t+\xi_n^t)^t$ reduces to the case $\mu=\varnothing$
exactly as with the usual crystal structure.
\end{proof}

\begin{proof}[Proof of Proposition \ref{Prop:comb_cond_holds}]
For large $n$ (namely, for $n\geqslant \lambda^t_1$) the diagram $C_{0,n}\lambda$ has exactly
$n$ rows. However, $C_{0,n}^*\mu$ has $\mu_1^t+n$ rows. It follows that
$C_{0,n}^*\mu\not\succeq C_{0,n}\lambda$, which shows that $(\mathfrak{C}_{\lambda,\mu})$ holds.
\end{proof}

\section{Extended quotients and their equivalences}\label{S_C_cat_equi}
\subsection{Setting}\label{SS_ext_quot_setting}
The KZ functor is not 0-faithful for certain values of parameters, see Section \ref{SS_KZ}.
For these values we need to take some intermediate quotient.
This will be a quotient functor associated to the projectives $P(\lambda)$
with $\lambda\in \mathcal{E}$, where the latter  is the union of crystal components that
depends on the choice of $e,s_1,\ldots,s_\ell$.

Let us introduce a notation. We write $\OCat^2_R$ for the Cherednik category $\mathcal{O}$ and $\OCat^1_R$
for the truncated affine parabolic category (that is a restricted  categorification).
We write
$\pi^i_R: \OCat^i_R\twoheadrightarrow \Hecke^{\bs}_{q,R}\operatorname{-mod}
(=\bigoplus_n \Hecke^{\bs}_{q,R}(n)\operatorname{-mod})$
for the quotient functors. When we omit the subscript $R$, it means that
we consider the categories and functors obtained by the base change to the closed point, $p$.

The definition of  $\mathcal{E}$ will be different for $e>2$ and for $e=2$.

Recall that $\Par^{(1)}_\ell$ denotes the union of all crystal components in
$\Par_\ell$ that contain an $\ell$-partition of 1. The singular $\ell$-partitions of 1
are as follows: for each residue $\alpha$
mod $e$, we take those of the $\ell$ one-box multipartitions whose content is congruent to $\alpha$ with the exception
of the minimal such multipartition (it lies in $\Par^{(0)}_\ell$). For example, if $\ell=4,e=3, \bs=(0,0,4,1)$, then
the multipartitions $(\varnothing, (1),\varnothing, \varnothing), (\varnothing, \varnothing, (1), \varnothing)$ are singular, while the other two single box multipartitions are not.

Now let us consider the case $e=2$.
Let $\nu$ be an $\ell$-multipartition
of $2$ with the following properties:
\begin{enumerate}
\item $\nu$ is not a column of length $2$.
\item $\nu$ is minimal satisfying (1).
\end{enumerate}
This multipartition $\nu$ is constructed as follows. Take the leftmost minimal $s_a$.
One of the boxes of $\nu$ will be the $(1,1)$-box in the $a$th diagram. If there is
$b<a$ with $s_b=s_a+1$, then the other box of $\nu$ will be the $(1,1)$ box in
the $b$th diagram (provided $b$ is minimal such).
Otherwise the other box of $\nu$ will be the $(1,2)$-box in the $a$th diagram.
It is straightforward to check that $\nu$ is singular.
For example, if $\ell=4$, and $\bs=(2,3,0,1)$, then $\nu=(\varnothing,\varnothing, (2),\varnothing)$,
and if $\bs=(2,3,1,0)$, then $\nu=(\varnothing,\varnothing,(1),(1))$.

Now let us give the definition of $\mathcal{E}\subset \Par_\ell$.

\begin{defi}\label{defi_E}
If $e>2$, we take $\mathcal{E}:=\Par^{(1)}_\ell$. If $e=2$, for $\mathcal{E}$, we take the union of
the component of $\nu$ with $\Par^{(1)}_\ell$.
\end{defi}

Let $\Cat^i_R$ denote  the quotient of $\OCat^i_R, i=1,2,$ defined by the projectives $P^i_R(\lambda)$
with $\lambda\in \mathcal{E}$. The inclusion $\Par^{(0)}_\ell\subset \mathcal{E}$ gives rise
to the quotient functor $\Cat^i_R\twoheadrightarrow \Hecke^{\bs}_{q,R}\operatorname{-mod}$
to be denoted by $\underline{\pi}^i_R$.

\subsection{Main result}\label{SS_ext_main}
Here is the main result of this section.


\begin{Prop}\label{Thm:ext_cat_equi}
There is an equivalence $\Cat^1_R\xrightarrow{\sim}\Cat^2_R$ intertwining
the functors $\underline{\pi}^1_R,\underline{\pi}^2_R$ and preserving the
labels of the projective objects.
\end{Prop}

The theorem  will follow if we show that
\begin{enumerate}
\item $\pi^1_R(P^1_R(\lambda))=\pi^2_R(P^2_R(\lambda))$ for any $\lambda\in \mathcal{E}$.
\item $\underline{\pi}^1_R$ is fully faithful on projectives.
\end{enumerate}
We remark that $\underline{\pi}^2_R$ is fully faithful on projectives because the KZ functor $\pi^2_R$ is.

Below we will see that (1) and (2) will follow if we check the following claims (a similar strategy was used
in \cite{LW}):
\begin{itemize}
\item[(a1)] $\pi^1_R(P^1_R(\lambda))\cong\pi^2_R(P^2_R(\lambda))$ for any singular $\lambda$ with $|\lambda|=1$.
\item[(b1)] $\pi^1_R(P^1_R(\nu))\cong\pi^2_R(P^2_R(\nu))$ (if $e=2$).
\item[(a2)] $\Hom_{\OCat^1}(P^1(\lambda),P^1(\mu))=\Hom_{\Hecke_q^{\bs}(1)}(\pi^1(P^1(\lambda)), \pi^1(P^1(\mu)))$ for  $\lambda,\mu$
with $|\lambda|=|\mu|=1$.
\item[(b2)] $\End_{\OCat^1}(P^1(\nu))=\End_{\Hecke_q^{\bs}(2)}(\pi^1(P^1(\nu)))$ (when $e=2$).
\end{itemize}
Here we write $P^1(\lambda),P^2(\lambda)$ for the specializations of $P^1_R(\lambda),P^2_R(\lambda)$.

More precisely, we will see that (a1) and (b1) imply (1), while (a2) and (b2) imply (2).

Conditions (a1) and (a2) are easy to check.

\begin{Lem}\label{Lem:a1_a2}
(a1) and (a2)  are true.
\end{Lem}
\begin{proof}
Any block in $\OCat^i(1)$ has
some number, say $d$, of standard objects that are linearly ordered, let $\lambda_1<\ldots<\lambda_d$
be their labels. The quotient morphism $\pi^i$ is given by the Hom from $P^i(\lambda_1)$.

The block is a block in a basic $\sl_2$-categorification,
see \cite[Section 4.3]{str}, of length $d$. It follows from the main result of \cite{LW} that the block
in $\OCat^i$  is equivalent to the block  in the BGG category $\mathcal{O}$ for $\gl_d$ with singularity
of type $\mathfrak{S}_{d-1}$.  So $\underline{\pi}^i$ is fully faithful on projectives.
This implies (a2).

To prove (a1) we need to deal with the deformed categories. To do this we point out that
we have an exact sequence $0\rightarrow P^i(\lambda_j)\rightarrow P^i(\lambda_{j-1})\rightarrow \Delta(\lambda_{j-1})
\rightarrow 0$.  This exact sequence deforms to $0\rightarrow P^i_R(\lambda_j)\rightarrow
P^i_R(\lambda_{j-1})\rightarrow \Delta^i_R(\lambda_{j-1})\rightarrow 0$. The $R$-module $\pi^i_R(\Delta^i_R(\lambda_j))$
free of rank $1$, where $\Hecke^{aff}_{q}(1)$ acts with an eigenvalue say $Q_j$, where $Q_1,\ldots,Q_d$
are pairwise different elements of $R^\times$ independent of $i$. Since $\pi^1_R(P^1_R(\lambda_1))\cong \pi^2_R(P^2_R(\lambda_1))$,
we prove by induction on $j$ that $\pi^1_R(P^1_R(\lambda_j))=\pi^2_R(P^2_R(\lambda_j))$.
\end{proof}

(b1) and (b2) are more complicated and will be proved in Section \ref{SS_deg1}.
To do that we will need an explicit construction of the objects $P^i_R(\nu)$
that will be carried in Section \ref{SS_constr_proj1}.

\subsection{Objects $Q^j_R(\nu)$}\label{SS_constr_proj1}
Let us write $\Delta^i_{A,R}(?)$ for standard objects in the level $1$ categories
(the truncated Kazhdan-Lusztig category or the category $\mathcal{O}$ for a type
A Cherednik algebra).

\begin{defi}\label{def_Q}
Let $Q^i_R(\nu)$ be the component of $\Delta^i_{A,R}(2)\dot{\otimes}\Delta^i_R(\varnothing)$  in the block
with two different residues. Similarly, let $R^i_R(\nu)$ be the component of
$\Delta^i_{A,R}(1^2)\dot{\otimes}\Delta^i_R(\varnothing)$ in the same block.
\end{defi}

\begin{Prop}\label{Prop:proj_coinc}
We have $Q^i_R(\nu)=P^i_R(\nu)$.
\end{Prop}
In the proof we will assume that there are elements with different residues mod 2 among $s_1,\ldots,s_\ell$.
The case when all residues are the same  is somewhat exceptional but easier and can be treated
similarly to what is done below in the proof.
\begin{proof}
The proof is in several steps.

{\it Step 1}. Note that $\pi^i_R(Q^i_R)$ is independent of $i$ and coincides with the block component
of $\Hecke^{\bs}_{q,R}(2)/(T+1)$. Note also that $\pi^i_R(R^i_R)=\pi^i_R(Q^i_R)$.
This is because the functor $\pi^i_R$ intertwines $\dotimes: \OCat^i_{A,R}\boxtimes
\OCat^i_R\rightarrow \OCat^i_R$ with the induction functor
$\Hecke_{q,R}\operatorname{-mod}\boxtimes \Hecke^{\bs}_{q,R}\operatorname{-mod}
\rightarrow \Hecke^{\bs}_{q,R}\operatorname{-mod}$. Note also that we have
a short exact sequence
\begin{equation}\label{eq:exact_seq} 0\rightarrow R^i_R(\nu)\rightarrow (F_0F_1\oplus F_1F_0)\Delta^i_R(\varnothing)
\rightarrow Q^i_R(\nu)\rightarrow 0.\end{equation}

{\it Step 2}. We claim that the $Q^i_R(\nu)$  admits a filtration with quotients of the form $\Delta^i_R(\lambda)$,
where $\lambda$ is a two-box multipartition with boxes of different residues that is not a column, each
occurring with multiplicity $1$. Similarly, $R^i_R(\nu)$  admits a filtration with quotients of the form $\Delta^i_R(\lambda)$, where $\lambda$ is a two-box multipartition with boxes of different residues that is not a row, each occurring with multiplicity $1$. For $i=1$, this follows from \cite[Proposition A2.6(b)]{VV}.
For $i=2$, this follows from the fact that the Bezrukavnikov-Etingof induction functors on the level of $K_0$
behave as the induction on the level of the groups.

{\it Step 3}. We claim that $R^i_R(\nu)$ is tilting. It is enough to prove this after the specialization
to $p$.

Let us check this for $i=1$. It is enough to check that $R^1(\nu)$  is tilting in the whole affine parabolic category
$\OCat^{\p}_{-e}$, not  in the truncation (indeed, the truncation is a highest weight subcategory).
Recall that $M_0\dotimes M_1$ is standardly filtered provided both $M_0,M_1$ are,
\cite[Corollary 7.3]{VV}. Now let $M_0$ be standard with $DM_0=M_0$,
where $D$ is the duality as in \cite[Section 2.6]{VV}, we can take $M_0=\Delta^1_{A}(1^2)$.
Note that $M_0\dotimes\bullet$ preserves the subcategory of costandardly filtered
objects.  This is a consequence of the fact that  the functors $M_0\dotimes\bullet$
and $DM_0\dotimes \bullet$ are biadjoint, \cite[Corollary 7.3]{VV}, combined
with the fact that $M_0\dotimes\bullet$ preserves standardly filtered objects.
Since $\Delta^1_R(\varnothing)$ is tilting, we see that $\Delta_{A}(1^2)\dotimes \Delta^1(\varnothing)$
is tilting as well. So  the object $R^1(\nu)$ is indeed tilting.

Now let us consider $i=2$. By \cite[Proposition 1.9]{Shan}, the Bezrukavnikov-Etingof restriction functors
$\Res_?$ preserves the subcategories of standardly filtered objects. By \cite[Conjecture 3.17]{BE}
proved in \cite{fun_iso},  the functors $\Res_?$ intertwine the naive duality functors
for Cherednik categories $\mathcal{O}$ introduced in \cite[Section 4.2]{GGOR}. So the functors
$\Res_?$ preserve the subcategories of costandardly filtered objects as well. By the adjointness,
the induction functors preserve the subcategories of standardly filtered objects and
of costandardly filtered objects. So $R^2(\nu)$ is tilting.

Note also that $Q^2(\nu)$ is projective. This is because $\Delta^2_A(2)$ is projective.

{\it Step 4}. Now we want to understand the structure of the objects $E_0 Q^i(\nu), E_1 Q^i(\nu),
E_0 R^i(\nu)$, $E_1 R^i(\nu)$. We will do the case of $E_0$, the other is analogous.
Without loss of generality we will assume that $s_a$ is even. The object $E_0 Q^1(\nu)$ has a nilpotent
endomorphism, denoted by $X$, that equals $X_2-1$. Recall that we write $\ell_1$ for the number of indexes $i$
with $s_i=0$, and $\ell_{-1}=\ell-\ell_1$.
Let $\lambda_1<\ldots<\lambda_{\ell_1}$ be all single box multipartitions with residue $1$. We claim that $E_0 Q^i(\nu) =E_0 R^i(\nu)=P^i(\lambda_1)\boxtimes \C[X]/(X^{\ell_1+1})$.

By Step 2, the modules $E_0 Q^i(\nu), E_0 R^i(\nu)$ are standardly filtered, their standard composition factors are
$\Delta^i(\lambda)$, where $\lambda$ is a single box with residue $1$, each standard occurs with multiplicity
$\ell_1+1$. We can compute $\pi^i(E_0 Q^i(\nu)), \pi^i(E_0 R^i(\nu))$: both  equal $E_0 \pi^i(Q^i(\nu))=\C[X_1,X_2]/(X_2-1)^{\ell_1+1}(X_1+1)^{\ell_{-1}}$.
Here the generator of $\Hecke_q^{\bs}(1)$ acts by $X_1$ and the endomorphism $X$ by $X_2-1$. We see that $\pi^i(E_0 Q^i(\nu))= \pi^i(E_0 R^i(\nu))$ is indecomposable
as a $\Hecke_q^{\bs}(1)\otimes \C[X]$-module and therefore  $E_0 Q^i(\nu), E_0 R^i(\nu)$ have no decompositions into direct summands preserved by $X$.
By (\ref{eq:exact_seq}), $E_0 Q^i(\nu)$ is the kernel
of $E_0(F_0 F_1\oplus F_1 F_0) \Delta^i(\varnothing)\twoheadrightarrow E_0 R^i(\nu)$.
The object $E_0(F_0 F_1\oplus F_1 F_0) \Delta^i(\varnothing)$ is a projective-injective and so is a direct sum of several copies of $P^i(\lambda_1)$. The object $E_0 R^i(\nu)$ is tilting. As  we have already observed in the proof of
Lemma \ref{Lem:a1_a2}, the block of $\OCat^i(1)$ with residue $1$
is equivalent to a block in a basic highest weight $\mathfrak{sl}_2$-categorification
in the sense of \cite{str}. By \cite[Sections 6.3,7.1]{str}, $T^i(\lambda_{j})$ has composition
series with standard subquotients $\Delta^i(\lambda_1),\ldots,\Delta^i(\lambda_j)$, each with multiplicity $1$.
Since $[E_0 R^1(\nu)]=(\ell_1+1)[T^i(\lambda_k)]$,
it follows that $E_0 R^1(\nu)$
is the sum of $\ell_1+1$ copies of $T^i(\lambda_k)=P^i(\lambda_1)$ and so is projective. Therefore
$E_0(F_0 F_1 \Delta^i(\varnothing)\oplus F_1 F_0 \Delta^i(\varnothing))\twoheadrightarrow E_0 R^i(\nu)$
splits and we see that $E_0 Q^i(\nu)$ is also projective, isomorphic to $P^i(\lambda_1)^{\ell_1+1}$.
The quotient morphism $\pi^i$ is fully faithful on the projectives and on the tiltings in $\OCat^i(1)$,
for example, because this is always a block in a GGOR category $\mathcal{O}$.

Since $\pi^i(E_0Q^i(\nu))=\pi^i(P^i(\lambda_1))\otimes \C[X]/(X^{\ell_1+1})$, we deduce from the full faithfulness
that $E_0 Q^i(\nu)=E_0 R^i(\nu)=P^i(\lambda_1)\boxtimes \C[X]/(X^{\ell_1+1})$.
%

{\it Step 5}. We claim that the objects $\pi^i(Q^i(\nu))=\pi^i(R^i(\nu)), Q^i(\nu),R^i(\nu)$
are all indecomposable. Assume the converse, let one of these objects decompose as $N\oplus N'$.
The decomposition $E_0N\oplus E_0 N'$ is $X$-stable. From Step 4, it follows that
$E_0 N=0$ or $E_0 N'=0$. Similarly, $E_1 N=0$ or $E_1 N'=0$. In $\Hecke^{\bs}_{q}(2)\operatorname{-mod}$,
there is no nonzero object annihilated by both $E_0,E_1$. In $\OCat^i(2)$, there is no standardly
filtered object with this property. We conclude that $E_0 N=0$ and $E_1 N'=0$. We already see
that this cannot happen for $\OCat^i(2)$: if $\lambda$ has two boxes in different partitions,
both $E_0\Delta^i(\lambda)$ and $E_1\Delta^i(\lambda)$ are nonzero. By our assumptions on the parity
of $s_1,\ldots, s_\ell$, $\Delta^i(\lambda)$  occurs in the composition series of $Q^i(\nu)$
and $R^i(\nu)$. This shows that $Q^i(\nu)$ and $R^i(\nu)$ are indecomposable.
Since $\pi^2$ is fully faithful on the  tiltings, we see that $\pi^2(R^2(\nu))$ is indecomposable.

By Step 2, $Q^2(\nu)=P^2(\nu)$.

{\it Step 6}.
We claim that
$R^1_R(\nu)=T^1_R(\nu^*)$, where $\nu^*$ is the maximal multipartition appearing among labels
of the standard subquotients in a filtration of $R^1(\nu)$ by standards. Explicitly, $\nu^*$ can be described
as follows. Recall the combinatorial duality from Section \ref{SS_crystals}. Let $\nu'$ be constructed
as $\nu$ but for the dual multi-charge ${\underline{s}}^\dagger$. Then we set $\nu^*:=(\nu')^\dagger$.
In particular, $\Delta^1_R(\nu^*)$ can be a realized as a subobject in a filtration of $R^1_R(\nu)$ with standard composition factors. Since $R^1_R(\nu)$ is indecomposable (Step 5)
and tilting (Step 3),  we see that  $R^1_R(\nu)=T^1_R(\nu^*)$.

{\it Step 7}.  Now we are ready to prove that the classes of $P^1(\nu)$ and $Q^1(\nu)$
in $K_0$ are the same.
Let $b^{\bs}_\lambda, B^{\bs}_\lambda$ be the elements of Uglov's dual and usual canonical bases, respectively.
Then $[L^1(\lambda)]=b^{\bs}_\lambda$, see \cite[Section 8.2]{VV}, while $[T^1(\lambda)]=B^{\bs}_\lambda$. The latter is
checked similarly to \cite[Section 8.2]{VV} using the well-known fact that the Ringel dual of the affine parabolic
category $\mathcal{O}$ on level $\kappa$ is the affine parabolic category category $\mathcal{O}$ on level $-\kappa$.
On the other hand, the BGG reciprocity implies that $[P^1(\lambda)]=(B^{\bs^\dagger}_{\lambda^\dagger})^{\dagger}$, where the external $\dagger$ means the map between Fock spaces
that sends the standard basis element labeled by $\mu$ to that indexed by $\mu^\dagger$. From the previous paragraph
it follows that $B^{\bs^{\dagger}}_{\nu^\dagger}$ is the class of the analog of $R^1(\nu)$ in the category
corresponding to the multi-charge $\bs^\dagger$. But that class coincides with $[Q^1(\nu)]^{\dagger}$.
This implies that $[Q^1(\nu)]=[P^1(\nu)]$.

%

{\it Step 8}. We claim that if $Q^1(\nu)\twoheadrightarrow \Delta^1(\lambda)$, then $\lambda=\nu$.
Assume the converse. Then we have an epimorphism $Q^1(\nu)\twoheadrightarrow \Delta^1(\lambda)\oplus \Delta^1(\nu)$.
This gives rise to an epimorphism $E_i Q^1(\nu)\twoheadrightarrow E_i\Delta^1(\lambda)\oplus
E_i \Delta^1(\nu)$. It is equivariant with respect to the endomorphism $X$ of both sides
that was mentioned in Step 4.
From the description of $E_i Q^1(\nu)$ given in Step 4 it follows
that the head of $E_iQ^1(\nu)$ viewed as an object in $\OCat^1(1)\boxtimes \C[X]\operatorname{-mod}$
 is simple. Therefore $E_i\Delta^1(\lambda)=0$ or $E_i \Delta^1(\nu)=0$.
In particular, both $\lambda$ and $\nu$ should be rows with $(1,1)$-boxes
of different residues. That box in $\lambda$ has to have residue $1$ (recall that we have assumed
that $s_a$ is even).

Let $\tilde{\lambda}$ be the multi-partition consisting of the $(1,1)$-boxes
in $\lambda$ and $\nu$. Note that $\tilde{\lambda}\leqslant \lambda'$ for any
$\lambda'$ that contains a $(1,1)$ box with residue $1$. Pick
a filtration with standard quotients going in order on the kernel of $Q^1(\nu)\twoheadrightarrow \Delta^1(\lambda)$.
Let $M$ be the quotient of this filtration containing $\Delta^1(\tilde{\lambda})$
as a subobject  and  not containing $\Delta^1(\lambda')$ with $\lambda'$ as above.
We see that $Q^1(\nu)\twoheadrightarrow M\oplus \Delta^1(\lambda)$ and $E_{0}M=
E_{0}\Delta^1(\tilde{\lambda})$.  It follows that $E_{0}Q^1(\nu)\twoheadrightarrow
E_{0}\Delta^1(\lambda)\oplus E_{0}\Delta^1(\tilde{\lambda})$.  Similarly to the previous paragraph, we get a contradiction.

{\it Step 9}. Finally, we are ready to prove an isomorphism $Q^1(\nu)\cong P^1(\nu)$.
Recall how $P^1(\nu)$ is obtained. Let us order the partitions $\nu'\leqslant \nu$,
$\nu_1=\nu,\nu_2,\ldots,\nu_s$ in such a way that $\nu_i<\nu_j$ implies $i>j$.
Let $P_k$ denote the maximal quotient of $P^1(\nu)$ filtered by $\Delta(\nu_i)$
with $i\leqslant k$. Then $P_{k+1}$ is included into an exact sequence
$$0\rightarrow \Ext^1(P_k,\Delta^1(\nu_{k+1}))\otimes \Delta^1(\nu_{k+1})\rightarrow P_{k+1}\rightarrow P_{k}\rightarrow 0.$$ We have seen in Step 7 that  $\Delta(\lambda)$ occurs  in $P^1(\nu)$ if and only
if $\lambda$ is not a column, and all multiplicities are $1$. Now $Q_k$ be the maximal
quotient of $Q^1(\nu)$ filtered with $\Delta^1(\nu_1),\ldots,\Delta^1(\nu_k)$.
We will prove that $P_k\cong Q_k$ by induction on $k$. The case $k=1$
is trivial. Now suppose that we know that $P_k\cong Q_k$ and want to prove that
$P_{k+1}\cong Q_{k+1}$. It is enough to prove that the extension
$$0\rightarrow \Delta(\nu_{k+1})\rightarrow Q_{k+1}\rightarrow Q_k\rightarrow 0$$
does not split. But this follows from Step 8.

The isomorphism $Q^1(\nu)\cong P^1(\nu)$ implies $Q^1_R(\nu)\cong P^1_R(\nu)$.
\end{proof}

%

\subsection{Proof of Proposition \ref{Thm:ext_cat_equi}}\label{SS_deg1}
In this section we complete the proof of Proposition \ref{Thm:ext_cat_equi}.


\begin{Lem}\label{Lem:b1_b2}
(b1) and (b2)  hold.
\end{Lem}
\begin{proof}
Let us start with (b1). By Proposition \ref{Prop:proj_coinc}, $P^i(\nu)=Q^i(\nu)$.
By Step 1 of the proof of that proposition, $\pi^i(Q^i(\nu)),i=1,2,$ both coincide with the direct summand
of $\Hecke_q^{\bs}(2)/(T+1)$, where $X_1X_2$ acts with generalized eigenvalue $-1$. Similarly, one sees
that $\pi^j_R(Q^j_R(\nu)), j=1,2,$ is a  direct summand in $\Hecke_{q,R}^{\bs}(2)/(T+1)$,
same for both $j$. This proves (b1).

Let us proceed to proving (b2). Recall, see (\ref{eq:exact_seq}),
that $Q^1(\nu)\subset (F_0F_1\oplus F_1F_0)\Delta^1(\varnothing)$. In particular, the socle
of $Q^1(\nu)$ does not contain simples annihilated by $\pi^1$. It follows that
$\End(Q^1(\nu))\hookrightarrow  \End(\pi^1(Q^1(\nu)))$.
So to prove (b2), thanks to the double centralizer property for $\pi^2$, it is enough to show that
$\dim \End(Q^1(\nu))=\dim \End(Q^2(\nu))$. By the BGG reciprocity, $\dim \End(P^i(\nu))$ equals $\sum_{\lambda}
 [P^i(\nu):\Delta^i(\lambda)][\Delta^i(\lambda):L^i(\nu)]=\sum_{\lambda} [P^i(\nu):\Delta^i(\lambda)]^2$.
The latter is the number of multipartitions of $2$ that are not column and does not depend on
the choice of $\OCat^i$.
\end{proof}

\begin{proof}[Proof of Proposition \ref{Thm:ext_cat_equi}]
Let us show that (a2) and (b2)  imply (2)  from Section \ref{SS_ext_main}. First of all,
we claim that (a2) implies
$$\Hom_{\OCat^i_R}(P^i_R(\lambda),P^i_R(\mu))=
\Hom_{\Hecke^{\bs}_{q,R}(1)}(\pi^i_R(P^i_R(\lambda)), \pi^i_R(P^i_R(\mu))), |\lambda|=|\mu|=1.$$
This is because the natural homomorphism from the left hand side to the right hand side  is an isomorphism
after specialization and so, since the right hand side is $R$-flat, is an isomorphism. Similarly, (b2)
implies $\End_{\OCat^i_R}(P^i_R(\nu))=\End_{\Hecke^{\bs}_{q,R}(2)}(\pi^i_R(P^i_R(\nu)))$.

For  $\lambda\in \mathcal{E}$, the module $P^i_R(\lambda)$
is a direct summand of $F^k \pi^i_R(P_R(\lambda^0))$ with $|\lambda^0|\leqslant 1$ or $\lambda^0=\nu$.
So  it is enough to prove that
\begin{equation}\label{eq:Homs}\Hom_{\OCat^i_R}(F^k P^i_R(\lambda^0), F^{k'}P^i_R(\mu^0))=
\Hom_{\Hecke^{\bs}_{q,R}(m)}(F^k\pi_R^i(P^i_R(\lambda^0)), F^{k'}\pi^i_R(P^i_R(\mu^0))),\end{equation}
where $m=|\lambda^0|+k=|\mu^0|+k'$. Using the biadjointness of $E$ and $F$, we can reduce  the proof
of (\ref{eq:Homs}) to showing that
$$\Hom_{\OCat^i_R}(P^i_R(\lambda),P^i_R(\mu))=\Hom_{\Hecke^{\bs}_{q,R}(|\lambda|)}(\pi^i_R(P^i_R(\lambda)), \pi^i_R(P^i_R(\mu))),$$ when one of $\lambda,\mu$ is singular. Using the biadjointness again, we reduce to the case when both $\lambda,\mu$ are singular.  This case has been established in the previous paragraph.

Let us show that (a1) and (b1) imply (1).
First, let us check that the sets $\{\pi^i_R P_R^i(\lambda), \lambda\in \mathcal{E}\},i=1,2,$ coincide. Indeed, since
the functors $\underline{\pi}^i_R, i=1,2,$ are fully faithful on the projectives,
both sets consist  precisely of the indecomposable summands of the modules of the form $F^n \pi_R^i (P_R^i(\lambda^0))$
with $|\lambda^0|\leqslant 1$ or $\lambda^0=\nu$ (the latter applies only to the
case $e=2$).
Now to check that $\pi^1_R P_R^1(\lambda)=\pi^2_R P_R^2(\lambda)$  one notices that the crystals are the same
on the level of labels (this follows from the observation that they are the same even for $\OCat^i$,
in that case, the coincidence of crystals follows from the main result of \cite{cryst}).
Also the singular labels agree by (a1). It follows that $\pi^1_R P_R^1(\lambda)=\pi^2_R P_R^2(\lambda)$
for all $\lambda\in \mathcal{E}$.
%
%
%
\end{proof}

\subsection{Summary}
We are going to summarize things that we have already proved and list things that we still need to prove.
By Proposition \ref{Thm:ext_cat_equi},
we have equivalent categories $\Cat^1_R\cong \Cat^2_R$ with quotient functors $
\overline{\pi}^1_R:\OCat^1_R\twoheadrightarrow \Cat^1_R,
\overline{\pi}_R:\OCat^2_R\twoheadrightarrow \Cat^2_R$ defined by the projectives
with labels in $\mathcal{E}$ for $\mathcal{E}$ from Definition \ref{defi_E}.
Below we will write $\Cat_R$ instead of $\Cat^i_R$.

To establish the asymptotic version of Conjecture \ref{Conj:VV}, it is sufficient to show the following.

\begin{Thm}\label{Thm:main}
There is a category equivalence $\OCat^1_R(n)\rightarrow \OCat^2_R(n)$
intertwining the quotient functors $\OCat^i_R(n)\twoheadrightarrow \Cat_R(n)$.
\end{Thm}

In the theorem we assume that $\OCat^1_R=\bigoplus_{i\leqslant N} \OCat^1_R(i)$
with $N\gg n$. Theorem \ref{Thm:main} will be proved in the next section.

\section{Proof of Theorem \ref{Thm:main}}\label{S_proof_complete}
In this section we prove Theorem \ref{Thm:main}.

\subsection{Functor $\overline{\pi}^2_R$ is $1$-faithful}\label{SS_0_faith}
Here we prove the following result.

\begin{Prop}\label{Prop:pi2_1faith}
The functor $\overline{\pi}^2_R:\OCat^2_R\twoheadrightarrow \Cat_R$ is $1$-faithful.
\end{Prop}
\begin{proof}
By Lemma \ref{Lem:1_faith}, it is sufficient to show that $\overline{\pi}^2$ is $0$-faithful.
The proof  is similar to that in \cite[Proposition 5.9]{GGOR} for the usual KZ functor.

Recall that to any module in $\OCat^2(n)$ we can assign its support that is a closed subvariety of $\C^n$.

\begin{Lem}\label{Lem:Par1_prop}
The set $\mathcal{E}\cap\Par_\ell(n)$ consists precisely of $\lambda$ with $\operatorname{codim}_{\C^n}\operatorname{Supp}L(\lambda)\leqslant 1$.
\end{Lem}
\begin{proof}[Proof of Lemma \ref{Lem:Par1_prop}]
The codimension is invariant on the crystal components, this follows from \cite[5.5]{cryst}. So it is enough
to assume that $\lambda$ is singular. Recall that the support is described by $2$ integers, see, for example,
\cite[Section 3.10]{shanvasserot}: a pair $(k,j)$ corresponds to the support $\{(x_1,\ldots,x_n)\}$, where we have $k$
zeroes and $j$  $e$-tuples of pairwise equal numbers.  The codimension equals $k+(e-1)j$. The partition
$\lambda$ is singular if and only if $k+ej=n$. So the codimension does not exceed $1$ if and only if
$e=2,k=0,j=1$ or $k\leqslant 1,j=0$. The second possibility holds precisely for singular $\lambda$
with $|\lambda|\leqslant 1$. It remains to show that the first possibility holds precisely for $\nu$.
For singular $\lambda$ with $|\lambda|=2$ the condition that $j=1$ is equivalent to $\operatorname{Res}^{G_2}_{\mathfrak{S}_2}L(\lambda)\neq 0$.
So this condition holds for $\lambda=\nu$ because $P^2(\nu)=Q^2(\nu)$,
see Proposition \ref{Prop:proj_coinc}. By  support considerations, if $\lambda$ is singular,
then $\operatorname{Res}^{G_2}_{\mathfrak{S}_2}L^2(\lambda)$ is a direct sum of several copies of $L_A^2(2)$.
So $\operatorname{Ind}_{\mathfrak{S}_2}^{G_2}\Delta_A^2(2)\twoheadrightarrow L^2(\lambda)$
(recall that $P^2_A(2)= \Delta^2_A(2)$). Equivalently, $P^2(\nu)\twoheadrightarrow L^2(\lambda)$ so $\lambda=\nu$.
\end{proof}

Using Lemma \ref{Lem:Par1_prop},  let us prove that $\overline{\pi}^2$ is $0$-faithful. This is equivalent to
$\operatorname{Ext}^i(L^2(\lambda),\Delta^2(\mu))=0$ for $i=0,1, \lambda\not\in \mathcal{E},\mu\in \Par_\ell$.
The case of $i=0$ is classical (it holds for all $\lambda\not\in \Par^0_\ell$). Now let $M$ be a non-trivial
extension of $L(\lambda)$ by $\Delta(\mu)$. Let $(\C^n)^{reg,1}$ be the open subspace in $\C^n$
obtained by removing all stratas (by the stabilizer in $G_n$) of codimension $>1$. Then we can
restrict $M$ to $(\C^n)^{reg,1}$, the restricted sheaf $M|_{(\C^n)^{reg,1}}$ is a module over the
restriction of $H_c$. The global sections  $\Gamma(M|_{(\C^n)^{reg,1}})$  again form an $H_c$-module. But since $\Delta^2(\mu)$ is a free coherent sheaf on $\C^n$ and $L$ is supported outside of $(\C^n)^{reg,1}$,  the composition $\Delta^2(\mu)\rightarrow M\rightarrow \Gamma(M|_{(\C^n)^{reg,1}})$ is an isomorphism. So the extension $0\rightarrow\Delta^2(\mu)\rightarrow M\rightarrow L^2(\lambda)\rightarrow 0$ is trivial.
\end{proof}
%

\subsection{Functor $\overline{\pi}^1_R$ is $0$-faithful}\label{SS:check_faithf_codim1}
Let $\p$ be a point in $R$.
Recall that we have linear functions $y_0=\kappa,y_i=\kappa s_i-i/\ell, i=1,\ldots,\ell$ on $\param$,
where $y_1,\ldots,y_\ell$ are defined up to a common summand.  Recall that
$y_0(p)=-\frac{1}{e}, y_i(p)=-\frac{s_i}{e}-\frac{i}{\ell}$.
We can view $y_0,y_i-y_j, 1\leqslant i,j\leqslant \ell$
as  elements of the residue field $\mathbf{k}_{\p}$ of $\p$.

We start by determining what type A Lie algebra acts on $\OCat^1_{\p}$.
Define an equivalence relation on $\{1,\ldots,\ell\}$ by setting $i\sim_{\p} j$
if $(y_i-p_i)-(y_j-p_j)=zy_0$ in $\mathbf{k}_{\p}$, where $z\in \Z$. Let $I^1_{\p},\ldots,I^s_{\p}$
denote the equivalence classes.

\begin{Lem}\label{Lem:act_type}
The Lie algebra acting on $\OCat^1_{\p}$ is determined as follows.
\begin{itemize}
\item Suppose that $y_0\neq -\frac{1}{e}$ in $\mathbf{k}_{\p}$. Then the algebra acting
on $\OCat^1_{\p}$ is $\gl_{\infty}^{\oplus s}$. The module $K_0(\OCat^1_{\p})$
is the exterior tensor product of the Fock spaces of levels $|I^1_{\p}|,\ldots, |I^s_{\p}|$,
where the basis in the $k$th Fock space, $k=1,\ldots,s,$ is indexed by the multipartitions
from $I^k_{\p}$.
\item Suppose that $y_0= -\frac{1}{e}$ in $\mathbf{k}_{\p}$. Then the algebra acting
on $\OCat^1_{\p}$ is $\mathfrak{sl}_{e}^{\oplus s}$. The description of $K_0(\OCat^1_{\p})$
repeats   the previous case.
\end{itemize}
\end{Lem}
\begin{proof}
This is completely analogous to \cite[Section 4.2, Proposition 4.4]{Shan}.
\end{proof}

\begin{Lem}\label{Lem:semisimple}
The following two conditions are equivalent:
\begin{enumerate}
\item The category $\OCat^1_{\p}$ is split semisimple and the quotient functor  $\pi^1_{\p}:\OCat^1_{\p}\twoheadrightarrow \Hecke^{\bs}_{q, \p}\operatorname{-mod}$ is an equivalence.
\item The following elements are nonzero in $\mathbf{k}_{\p}$: $y_0+\frac{1}{e}, (y_i-p_i)-(y_j-p_j)-zy_0$ for $z\in \Z$
and $1\leqslant i<j\leqslant \ell$.
\end{enumerate}
\end{Lem}
\begin{proof}
Condition (1) is equivalent to $\mathbf{k}_{\p}\otimes_R \Hecke^{\bs}_{q,R}\operatorname{-mod}$ being split semisimple.
This is equivalent to (2) by \cite{ArikiKoike}.
\end{proof}

\begin{Prop}\label{Prop:codim_1_faith}
Let $\p$ be a point of codimension $1$. Then the functor $\overline{\pi}_{\p}$ is $(-1)$-faithful.
\end{Prop}
\begin{proof}
By Lemma \ref{Lem:semisimple}, we only need to check the cases when
\begin{itemize}
\item[(i)]  $\p$ is generic
with $y_0=-\frac{1}{e}$.
\item[(ii)] $\p$ is generic with $(y_i-p_i)-(y_j-p_j)=zy_0$.
\end{itemize}

Let us consider (i). By Lemma \ref{Lem:act_type}, we have an $\hat{\sl}^\ell_e$-action on $\OCat^1_{\p}$.
The (-1)-faithfulness of $\pi_\p$ is deduced from a 
direct analog of Proposition \ref{Prop:-1_faith} (for the $\hat{\mathfrak{sl}}_e^\ell$-crystal on
$\Par_1^\ell$) combined with Proposition \ref{Prop:-1_faith}.

Let us proceed to (ii).  By Lemma \ref{Lem:act_type}, here we have $\ell-1$-copies of $\gl_{\infty}$
acting.  In this case each weight space
for the $\ell-2$ copies of $\gl_{\infty}$ with level $1$ actions is a highest weight
$\gl_{\infty}$-categorification
of a level 2 Fock space. Let the multi-charge of that Fock space be $(t_1,t_2)$.
So the standard objects are parameterized by 2 partitions $(\lambda^{(1)},\lambda^{(2)})$.
For each integer $c$ we may have not more than 2 addable/removable boxes with  shifted content
equal to $c$, the box in $\lambda^{(2)}$ is bigger than that in $\lambda^{(1)}$ (i.e., precedes it in the signature).

Now let $\lambda=(\lambda^{(1)},\lambda^{(2)})$ be a singular bi-partition, and $\mu=(\mu^{(1)},\mu^{(2)})$ be a cosingular one, lying in the same block.
The condition that $\lambda$ is singular is equivalent to $\lambda^{(1)}=\varnothing$
and  $\lambda^{(2)}$ has only one removable box, with shifted content equal to $t_2$. Similarly,
$\mu$ is cosingular if and only if $\mu^{(2)}=\varnothing$ and $\mu^{(1)}$ has a single removable box with shifted content
equal to $t_1$. In particular, we see that $\lambda>\mu$. So $(\mathfrak{C}_{\lambda\mu})$
holds (for $w=1$).

From Proposition \ref{Prop:-1_faith}, we deduce that $\overline{\pi}^1_{\p}$ is (-1)-faithful.
\end{proof}

\subsection{Completing the proof of Theorem \ref{Thm:main}}
\begin{proof}[Proof of Theorem \ref{Thm:main}]
The functors $\overline{\pi}^i_R:\OCat^i_R\twoheadrightarrow \Cat_R$ are equivalences
after the base change to $\operatorname{Frac}(R)$ because the functors $\pi^i_R$
are so. We will apply Theorem \ref{Thm:Thm_equi} to the projectives
$\overline{P}^i_R:=\bigoplus_{\lambda\in \mathcal{E}}P_R^i(\lambda)$
and $P^i_R=:\bigoplus_{\lambda\in \mathcal{P}^0_\ell}P_R^i(\lambda)$.
Condition (i) was verified  already in Section \ref{S_cat}.
Condition (ii) follows from Proposition \ref{Prop:pi2_1faith}.
Condition (iii) follows from Proposition \ref{Prop:codim_1_faith}
combined with Proposition \ref{Prop:faith}. Condition (iv)
follows from Proposition \ref{Prop:Cher_proj_emb}. Now Theorem \ref{Thm:Thm_equi}
implies Theorem \ref{Thm:main}.
\end{proof}


\section{Appendix. Another proof of the equivalence for $\ell=1$}\label{S_Cat_Sch}

For $\ell=1$,  the parabolic affine category $\mathcal{O}$ (the Kazhdan-Lusztig category)
is equivalent to the category of modules over Lusztig's form of the quantum group for $\gl_m$.
So the truncation $\OCat^{\g}_{-e}(n)$ is the category of modules over the  $q$-Schur algebra $S_q(m,n)$ for $m\geqslant n$. The categories for different $m\geqslant n$ are naturally identified.
For $m'>m$, the embedding $\bigoplus_{n=0}^m S_q(m,n)\operatorname{-mod}\hookrightarrow
\bigoplus_{n=0}^{m'} S_q(m',n)\operatorname{-mod}$ is compatible with the
restricted $\hat{\mathfrak{sl}}_e$-actions. We set $\mathcal{O}^S(n):=
S_q(n,n)\operatorname{-mod},\mathcal{O}^S:=\bigoplus_{n\geqslant 0}\mathcal{O}^S(n)$.
This is a highest weight categorification of the level one Fock space $\mathcal{F}_e$.
In particular,  $\OCat^\g_{-e}(\leqslant n)$ can be embedded into the genuine $\hat{\sl}_e$-categorification.

The category $\OCat^{S}(n)$ can also be described in a different way:
as the category of right modules over the endomorphism algebra of a certain  $\mathcal{H}_q(n)$-module,
where $\mathcal{H}_q(n)$ denotes the Hecke algebra of type A.
 The module we need is the sum of all indecomposables  $\mathcal{H}_q(n)$-modules that are induced from the trivial module over the product $\mathcal{H}_q(\lambda):=\mathcal{H}_{q}(\lambda_1)\boxtimes\mathcal{H}_q(\lambda_2)\boxtimes\ldots\boxtimes \mathcal{H}_q(\lambda_k)$ for all partitions $\lambda$ of $n$, we denote such modules by $\operatorname{Ind}_\lambda^{\mathcal{H}}(\operatorname{triv})$. Let $\pi^{S}$ denote the quotient functor $\OCat^{S}(n)\twoheadrightarrow \mathcal{H}_q(n)\operatorname{-mod}$. Then we can describe the image of $P(\lambda)$ under $\pi^{S}$: it is the only indecomposable direct summand $P_\lambda$ of $\operatorname{Ind}_\lambda^{\mathcal{H}}(\operatorname{triv})$ that does not appear in $\operatorname{Ind}_\mu^{\mathcal{H}}(\operatorname{triv})$ for any $\mu<\lambda$ in the dominance ordering.
We remark that, by the construction, the functor $\pi^S$ is fully faithful on projectives.

\begin{Prop}
We have an equivalence $\mathcal{O}_\kappa(n)\xrightarrow{\sim}\mathcal{O}^S(n)$ that maps
$\Delta(\lambda)$ to $\Delta^S(\lambda)$.
\end{Prop}
\begin{proof}
Note that both the KZ functor and $\pi^S$ are fully faithful on projectives. So it is enough
to show that the images of $P(\lambda)$ and $P^S(\lambda)$ in $\mathcal{H}_q(n)\operatorname{-mod}$
coincide.

Let $\lambda=(\lambda_1,\ldots,\lambda_k)$ be a partition of $n$. Define an object $I\Delta(\lambda)\in \mathcal{O}_\kappa(n)$
as the image of $\Delta(\lambda_1)\boxtimes\Delta(\lambda_2)\boxtimes\ldots\boxtimes \Delta(\lambda_k)$
(an object in the category $\mathcal{O}$ for $\mathfrak{S}_{\lambda_1}\times\ldots\times \mathfrak{S}_{\lambda_k}$)
under the Bezrukavnikov-Etingof induction functor. Note that $I\Delta(\lambda)$ is projective
and that $KZ(I\Delta(\lambda))=\operatorname{Ind}_\lambda^{\mathcal{H}}(\operatorname{triv})$.
This is because the KZ functors intertwine the inductions and, by
\cite[Corollary 6.10]{GGOR}, $KZ(\Delta(\lambda_i))=\operatorname{triv}_{\lambda_i}$.

Recall that on the level of $K_0$ the Bezrukavnikov-Etingof
induction functor coincides with the induction for groups,
\cite[Propostion 3.14]{BE}.
The multiplicity of  $\Delta(\mu)$ in $I\Delta(\lambda)$
coincides with the number of semi-standard Young tableaux with $\lambda_i$
entries equal to $i$ and shape $\mu$. So $\Delta(\lambda)$ appears
in $I\Delta(\lambda)$ with multiplicity $1$ and if $\Delta(\mu)$
appears in $I\Delta(\lambda)$, then $\lambda\leqslant \mu$. It follows
that $I\Delta(\lambda)=P(\lambda)\oplus \bigoplus_{\mu<\lambda} P(\mu)^{\oplus a_\mu}$.
From here and the description of $\pi^S(P^S(\lambda))$ above, we deduce
that $\operatorname{KZ}(P(\lambda))=\pi^S(P^S(\lambda))$.
\end{proof}

\end{document}